\documentclass[11pt,reqno]{amsart}      %use amsart package with 11pt font
\usepackage{cite}                       %Makes citations fancier
\usepackage{amssymb,amsthm}             %get some extra symbols  and
\usepackage{calc}                       %graphics packages (calc is used by
\usepackage{graphicx}

\numberwithin{figure}{section}          %tie figure numbering to
\numberwithin{equation}{section}        %tie equation numbering to
\renewcommand{\Re}{{\mathbb R}}         %the real numbers
\newcommand{\NatNum}{{\mathbb N}}       %the natural numbers
\newcommand{\half}{\frac{1}{2}}         %half
\newcommand{\third}{\frac{1}{3}}        %third
\newcommand{\quart}{\frac{1}{4}}        %quart
\DeclareMathOperator*\esssup {ess \, sup}  %essential supremum
\theoremstyle{plain}
\newtheorem{thm}{Theorem}[section]
\newtheorem{cor}[thm]{Corollary}

\newtheorem{lemma}[thm]{Lemma}
\newtheorem{definition}[thm]{Definition}

\newtheorem{remark}[thm]{Remark}
\title[Optimal Control for Ginzburg-Landau]{Convergence Rates for an
  Optimally Controlled Ginzburg-Landau equation}
\author{Mattias Sandberg}
\address{CMA, University of Oslo, P.O. Box 1053 Blindern, 0316 Oslo, Norway
}
\email{mattias.sandberg@cma.uio.no}
\begin{document}
\begin{abstract}
An optimal control problem related to the probability of transition
between stable states for a thermally driven Ginzburg-Landau equation
is considered. The value function for  the optimal control problem
with a spatial discretization is shown to converge quadratically to
the value function for the original problem. This is done by using that the
value functions solve similar Hamilton-Jacobi equations, the equation for the
original problem being defined on an infinite dimensional Hilbert
space. Time discretization is performed using the Symplectic Euler
method. Imposing a reasonable condition this method is shown to be
convergent of order one in time, with a constant independent of the spatial discretization.
\end{abstract}
\maketitle

\tableofcontents
%\listoffigures
\section{Introduction}
We shall in this paper study the convergence of the Symplectic Euler
scheme for approximating optimal control of the real Ginzburg-Landau
equation. This follows the work developed in \cite{Sandberg-Szepessy}, where a
convergence result for the value function to an optimally controlled
ODE is obtained using the corresponding Hamilton-Jacobi equation. As
there exists a rigorous theory also for infinite-dimensional
Hamilton-Jacobi equations, developed by M.\ Crandall and P.-L. Lions
\cite{Crandall-Lions1,Crandall-Lions2,Crandall-Lions3,Crandall-Lions4,Crandall-Lions5,Crandall-Lions6,Crandall-Lions7}, it is possible to perform a convergence analysis
for a spatial discretization of an optimally controlled PDE, using
that the value function is a viscosity solution to an
infinite-dimensional Hamilton-Jacobi equation. In this paper the
analysis is performed for the specific problem at hand, but hopefully
the analysis is clear enough to make adaptations to other circumstances
(fairly) easy.

Consider the stochastic PDE
\begin{equation}\label{eq:SPDE}
\varphi_t=\delta\varphi_{xx}-\delta^{-1}V'(\varphi)+\sqrt{\varepsilon}\eta,
\quad \text{in } [0,T]\times[0,1],
\end{equation}
where $\delta$ is a positive number and $\eta$ is white noise in two dimensions; this means that $\eta$
is a random Gaussian distribution with zero mean and covariance
\begin{equation*}
E\big(\eta(x,t),\eta(x',t')\big)=\bar\delta(x-x')\bar\delta(t-t'),
\end{equation*}
where $E$ denotes the expectation and $\bar\delta$ the Dirac delta distribution. The ``state'' variable $\varphi$
satisfies the Dirichlet boundary conditions
\begin{equation*}
\varphi(t,0)=\varphi(t,1)=0,
\end{equation*}
and $V$ is the ``double-well'' potential
\begin{equation*}
V(\varphi)=\quart (\varphi^2-1)^2;
\end{equation*}
see Figure \ref{fig:V}. In one space dimension the stochastic PDE
(\ref{eq:SPDE}) makes sense, as existence and uniqueness of solutions
can be proved. Taking $\varepsilon=0$, the solutions to
(\ref{eq:SPDE}) generically approach one of the two stable critical
points, $\varphi_+$ or $\varphi_-$, (see Figure \ref{fig:phiplusphiminus}), which constitute
minima to the energy
\begin{equation*}
\int_0^1 \big( \frac{\delta}{2}\varphi_x^2 + \delta^{-1} V(\varphi)\big)dx.
\end{equation*}
\begin{figure}
\centering
\includegraphics[width=0.5\textwidth]{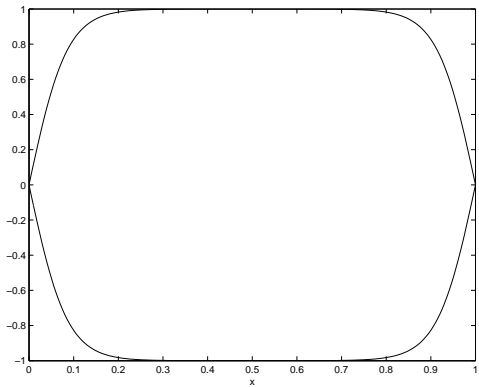}
\caption{Upper curve: $\varphi_+$. Lower curve: $\varphi_-$.}
\label{fig:phiplusphiminus}
\end{figure}
With a small $\varepsilon$, the solutions to (\ref{eq:SPDE}) spend
most of the time in the vicinity of either $\varphi_+$ or $\varphi_-$, but
as rare events make the transition from one to the other. The equation
(\ref{eq:SPDE}) may therefore be taken as a model for thermally driven
phase transitions, nucleations, etc.

The probability of jumping from $\varphi_+$ to $\varphi_-$ in the finite
time $T$ is related to the action functional
\begin{equation}\label{eq:actionfunctional}
I(\varphi)=\half\int_0^T\int_0^1 \big(\varphi_t - \delta\varphi_{xx} +
\delta^{-1} V'(\varphi)\big)^2 dx\, dt.
\end{equation}
Introduce the probability $P_T$ that a solution $\varphi$ to
(\ref{eq:SPDE}), with $\varphi(0)=\varphi_+$, satisfies $\varphi(T)\in S$, where
$S$ is an open subset of the set of continuous functions on the
spatial interval
$[0,1]$. Theory of large deviations  in \cite{Faris-Jona-Lasinio} gives that 
\begin{align*}
& -I(S) \leq \liminf_{\varepsilon \rightarrow 0}
 \varepsilon\log P_T \\
\intertext{and}
&\limsup_{\varepsilon \rightarrow 0}\varepsilon\log P_T \leq
 -I(\bar S)\\
\intertext{where}
&I(S)=\inf I(\varphi),
\end{align*}
with the infimum in the last equality taken over all continuous
functions $\varphi$ in $[0,T]\times[0,1]$ starting at $\varphi_+$ and ending
in $S$, and where $\bar S$ is the closure of $S$. By taking $S$ a small neighborhood of $\varphi_-$ we can
for small $\varepsilon$ approximate the probability of transition from
$\varphi_+$ to $\varphi_-$ with
\begin{equation*}
P_T \approx e^{-I(S)/\varepsilon}.
\end{equation*}

In \cite{Weinan-Ren-Vanden-Eijnden} the minimization of (\ref{eq:actionfunctional}) is
performed for $\varphi(0)=\varphi_+$ and $\varphi(T)=\varphi_-$ using optimization
of a finite difference approximation.
In this paper $\varphi(T)$ will not be held fixed, but
instead a penalty cost at the final time is added to the functional
(\ref{eq:actionfunctional}) in order to force the solution to end up
close to $\varphi_-$. The optimal control problem which will be
considered here is the following. Minimize, over all $\alpha \in
L^2\big(0,T;L^2(0,1)\big)$, the value $v_{\varphi_+,0}(\alpha)$, where
the functional $v$ is 
defined by
\begin{equation}\label{eq:costfunctional}
v_{\varphi_0,t_0}(\alpha)=\int_{t_0}^T h\big(\alpha(t)\big)dt+g\big(\varphi(T)\big),
\end{equation}
and where $\varphi$ is a mild solution to 
\begin{equation}\label{eq:flow}
\varphi_t=\delta \varphi_{xx}-\delta^{-1} V'(\varphi)+\alpha, \quad \varphi(t_0)=\varphi_0.
\end{equation}
In order to define mild solutions we denote by $S(t)$ the contraction
semigroup of bounded linear operators on $L^2(0,1)$ generated by
$\delta\, d^2/dx^2$ defined on $H^1_0(0,1) \cap H^2(0,1)$. A mild solution
to \eqref{eq:flow} is a function 
$\varphi \in C(t_0,T;L^2)$ such
that, for all $t_0 \leq t \leq T$,
\begin{equation}\label{eq:phimild}
\varphi(t)=S(t-t_0)\varphi_0 + \int_{t_0}^T S(t-s) \big(
-\delta^{-1}V'(\varphi(s))+\alpha(s) \big) ds.
\end{equation}
In the appendix existence and uniqueness of weak solutions in
$C(t_0,T;H_0^1)$ of \eqref{eq:flow} is proved when the starting
position $\varphi_0 \in H_0^1(0,1)$. Such weak solutions are
also mild solutions (this can be seen by using  e.g.\ the calculation on page 105 in \cite{Pazy}). Furthermore,
with $\alpha$ bounded in $L^2(t_0,T;L^2)$, the weak solution is
bounded in $C(t_0,T;H_0^1)$. Hence, the potential $V$ may be changed
outside an interval $[-s,s]$ without changing the result of the
optimal control problem. For simplicity, we shall henceforth use the
potential $\tilde V$ in Figure~ \ref{fig:V} and quickly change
notation, so that we let $V \equiv \tilde V$, i.e.\ $V$ is given by
the dashed line. Letting the transition from the interval $[-s,s]$
to the outside be a smooth one we can assume that arbitrarily many
derivatives of $V$ are bounded. 
\begin{figure}
\centering
\includegraphics[width=0.5\textwidth]{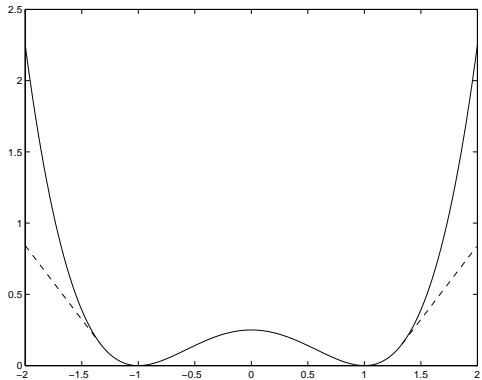}
\caption{The original potential $V$ is drawn with the solid line. The
  modified potential $\tilde V$ coincides with $V$ in the interval
  $[-s,s]$ and is drawn with dashed lines outside that interval.}
\label{fig:V}
\end{figure}
When the function $V'$ is bounded in supremum-norm and the control,
$\alpha$, is bounded in $L^2(t_0,T;L^2)$, uniqueness of mild solutions
to \eqref{eq:flow} holds; see \cite{Cannarsa-Frankowska1}. For
starting positions $\varphi_0 \in H_0^1(0,1)$  it therefore holds that
mild solutions and weak solutions are the same, and for the analysis
either concept of solution may be used.

The running cost, $h$, corresponds to the action functional
(\ref{eq:actionfunctional}) as
\begin{equation}\label{eq:runningcost}
h(\alpha)=||\alpha||_{L^2(0,1)}^2/2,
\end{equation}
and the final cost is the squared $L^2$ distance from $\varphi_-$,
\begin{equation}\label{eq:finalcost}
g(\varphi)=K||\varphi-\varphi_-||_{L^2(0,1)}^2,
\end{equation}
where $K$ is a constant large enough to force $\varphi(T)\approx
\varphi_-$. We denote by $u$ the \emph{value function}, i.e.\ the best
possible value of (\ref{eq:costfunctional}) for each starting position
$(\varphi_0,t_0)$:
\begin{equation}\label{eq:valuefunction}
u(\varphi_0,t_0)=\inf \big\{v_{\varphi_0,t_0}(\alpha)\ :\ \alpha \in L^2\big(t_0,T;L^2(0,1)\big)\big\}
\end{equation}

\emph{Notation:} We henceforth let $||\cdot ||$ and $(\cdot,\cdot)$ be
the $L^2$ norm and inner product on $(0,1)$, and $|\cdot|$ be the
supremum norm on $\Re$.
The Dirac delta distribution will be denoted $\bar\delta$, as $\delta$
is used as the diffusivity constant.

\emph{Outline:} Section \ref{sec:Preliminaries} contains some facts
regarding the value function, which are applied in Section
\ref{sec:spatialdiscr} when the error from the spatial discretization
is established. In Section \ref{sec:timediscr} convergence of the time
discretization using the Symplectic Euler method is examined. Under a
reasonable condition, this method is shown to be convergent of order
one, with a constant independent of the spatial
discretization. Numerical results with examples of the convergence
rate for discretization in both space and time  is given in Section~ \ref{sec:NumRes}.  
\section{Preliminaries}\label{sec:Preliminaries}
This section contains some results which will be needed when the spatial
discretization error bound is established in Section
\ref{sec:spatialdiscr}. We start with a theorem about the boundedness
of optimal controls. 
\begin{thm}\label{thm:boundedcontrol}
For all $\varphi_0 \in H_0^1(0,1)$ and $0\leq t_0 \leq T$ the value
function $u$ satisfies
\begin{equation*}
u(\varphi_0,t_0)=\inf \big\{v_{\varphi_0,t_0}(\alpha)\ :\ ||\alpha||_{L^\infty(t_0,T;L^2)} \leq E ||\varphi_0||
+ F\big\}
\end{equation*}
where the constants $E$ and $F$ depend on $\delta$, $K$, $\varphi_-$,
$T$, $|V'|$ and $|V''|$, but not on $\varphi_0$ and $t_0$.
\end{thm}
\begin{proof}
It is first shown that with $\alpha(t)= 0$, for all $t$, the state variable at
the terminal time, $\varphi(T)$, is bounded in $L^2$ by a constant which 
depends on the starting point $\varphi_0$. This can be done by taking the inner product
with $\varphi$ in (\ref{eq:flow}), using $||\varphi_x||^2 \geq
0$, and noting that the function $t \mapsto ||\varphi(t)||^2$ is
absolutely continuous with 
$(\varphi,\varphi_t)=\frac{d}{dt}||\varphi||^2/2$ almost everywhere in
$[t_0,T]$. Hence
\begin{equation*}
\frac{d}{dt}||\varphi||^2/2 \leq -\delta^{-1} \big(V'(\varphi),\varphi\big)
\leq \delta^{-1}|V'| \cdot ||\varphi||,
\end{equation*} 
almost everywhere in $[t_0,T]$, and thereby
\begin{equation}\label{eq:fnd}
\frac{d}{dt}||\varphi|| \leq \delta^{-1}|V'|.
\end{equation}
By the fact that $\varphi_t$ is bounded in $L^2(t_0,T;L^2(0,1))$ (see
\cite{Evans}), it follows that the function $t \mapsto ||\varphi(t)||$
is absolutely continuous, and therefore \eqref{eq:fnd} implies that 
$||\varphi(T)||$ is bounded by $||\varphi_0|| + \delta^{-1}|V'|T$. Hence the final
cost, $g\big(\varphi(T)\big)$, is bounded in terms of the starting position:
\begin{multline*}
g\big(\varphi(T)\big) =K ||\varphi(T)-\varphi_-||^2 \leq 2K \big( ||\varphi(T)||^2
+ ||\varphi_-||^2 \big) \\
\leq 4K||\varphi_0||^2 + 4K\delta^{-2}|V'|^2 T^2 +
2K||\varphi_-||^2 =: M.
\end{multline*} 
It therefore holds that
\begin{equation*}
u(\varphi_0,t_0)=\inf \big\{v_{\varphi_0,t_0}(\alpha)\ :\
||\alpha||_{L^2(t_0,T;L^2)}^2 \leq 2M \big\}.
\end{equation*}
For all $\alpha$ bounded by $\sqrt{2M}$ in $L^2(t_0,T;L^2)$ we have
that $\varphi(T)$ is bounded, again by taking the inner product with
$\varphi$ in (\ref{eq:flow}):
\begin{equation*}
\half \frac{d}{dt}||\varphi||^2 \leq \delta^{-1}|V'| \cdot ||\varphi|| +
||\alpha||\cdot ||\varphi||,
\end{equation*}
which implies
$\frac{d}{dt} ||\varphi|| \leq \delta^{-1}|V'| + ||\alpha||$
and so
\begin{multline*}
||\varphi(T)|| \leq ||\varphi_0|| + \delta^{-1}|V'|T + \int_{t_0}^T
  ||\alpha||dt \\
\leq ||\varphi_0|| + \delta^{-1}|V'|T +
  \sqrt{T}||\alpha||_{L^2(t_0,T;L^2)} \\
\leq ||\varphi_0|| +
  \delta^{-1}|V'|T +\sqrt{2T}\sqrt{M} \leq E ||\varphi_0|| +F,
\end{multline*}
for some constants $E$ and $F$ which do not depend on $\varphi_0$ and $t_0$.

It shall now  be proved that changing the control $\alpha$ a small
amount changes the state $\varphi$ a small amount. We shall therefore
compare two solutions, $\varphi^1$ and $\varphi^2$, both starting at
$(\varphi_0,t_0)$, such that $\varphi^1$ solves (\ref{eq:flow}) with control
$\alpha^1$ and $\varphi^2$ with control $\alpha^2$. Subtract the two
evolution equations and take the inner product with $\varphi^1-\varphi^2$ to
obtain
\begin{multline*}
\half\frac{d}{dt}||\varphi^1-\varphi^2||^2+\delta||\varphi^1_x-\varphi^2_x||^2
= \\
\delta^{-1}(-V'(\varphi^1)+V'(\varphi^2),\varphi^1-\varphi^2)+(\alpha^1-\alpha^2,\varphi^1-\varphi^2)
\end{multline*}
which, by the boundedness of $V''$, entails
\begin{equation*}
\frac{d}{dt}||\varphi^1-\varphi^2|| \leq \delta^{-1}|V''|\cdot ||\varphi^1-\varphi^2||+||\alpha^1-\alpha^2||.
\end{equation*}
By Gr\"onwall's lemma we therefore have that
\begin{equation*}
||\varphi^1(T)-\varphi^2(T)|| \leq \exp (\delta^{-1} |V''|T)
  \int_{t_0}^T ||\alpha^1-\alpha^2||dt,
\end{equation*}
so, provided $\alpha^1$ and $\alpha^2$ are both bounded by $\sqrt{2M}$
in $L^2(t_0,T;L^2)$, the difference in terminal cost has the following bound:
\begin{equation}\label{eq:gdifference}
\begin{split}
|g\big(\varphi^1(T)\big)-g\big(\varphi^2(T)\big)| &= K \big|
 ||\varphi^1(T)-\varphi_-||^2-||\varphi^2(T)-\varphi_-||^2\big| \\
&=K\big|(\varphi^1(T)+\varphi^2(T)-2\varphi_-,\varphi^1(T)-\varphi^2(T))\big| \\
& \leq 2K(E||\varphi_0||+F+||\varphi_-||) \cdot ||\varphi^1(T)-\varphi^2(T)|| \\
& \leq R \int_{t_0}^T ||\alpha^1-\alpha^2||dt, \\
\intertext{where}
R&=2K \exp (\delta^{-1} |V''|T)(E||\varphi_0||+F+||\varphi_-||) \\
&=:
E'||\varphi_0|| +F',
\end{split}
\end{equation}
with the constants $E'$ and $F'$ independent of $\varphi_0$ and $t_0$.
Let now $\alpha^1$ be a control bounded by $\sqrt{2M}$ in $L^2(t_0,T;L^2)$ and
let $\alpha^2$ be a modification of $\alpha^1$:
\begin{equation*}
\alpha^2(t)=
\begin{cases}
\alpha^1(t), & \text{for all $t$ such that $||\alpha^1(t)|| \leq 2R$}, \\
0, & \text{otherwise.}  
\end{cases}
\end{equation*}
The difference in the terminal cost thus has the  bound
\begin{equation*}
|g(\varphi^1(T))-g(\varphi^2(T))| \leq R \int_{\{t:||\alpha^1(t)||>2R \}}
 ||\alpha^1|| dt,
\end{equation*}
while the difference in running cost is 
\begin{equation*}
\int_{\{t:||\alpha^1(t)||>2R \}} \frac{||\alpha^1||^2}{2} dt \geq
R\int_{\{t:||\alpha^1(t)||>2R \}} ||\alpha^1||dt,
\end{equation*}
so
\begin{equation*}
v_{\varphi_0,t_0}(\alpha^2) \leq v_{\varphi_0,t_0}(\alpha^1).
\end{equation*}
Hence for any control bounded in $L^2(t_0,T;L^2)$ there is another
control, bounded by $2E'||\varphi_0||+2F'$ in $L^\infty(t_0,T;L^2)$, which gives a smaller or equal value functional.
\end{proof}
With Theorem \ref{thm:boundedcontrol} the theory in
\cite{Cannarsa-Frankowska1} may be used, which establishes existence
of optimal controls to $u$ in \eqref{eq:valuefunction}. We state this
in a corollary.
\begin{cor}\label{cor:existence}
For each $(\varphi_0,t_0) \in H_0^1(0,1) \times [0,T]$ there exists a
minimizer $\alpha$, bounded in $L^\infty(t_0,T;L^2)$, in \eqref{eq:valuefunction}.
\end{cor}  
\begin{proof}
Use Theorem \ref{thm:boundedcontrol} and \cite{Cannarsa-Frankowska1}.
\end{proof}
Theorem \ref{thm:boundedcontrol} is also used when proving Theorem
\ref{thm:semiconcave} about semiconcavity. In
\cite{Cannarsa-Frankowska2} a theorem on semiconcavity on $L^2(0,1)
\times [0,T)$ is established. This result could have been used in this
    paper, but as only the weaker result of semiconcavity on
    $H^1_0(0,1) \times [0,T]$ is needed for our purposes, an easier
    proof is given for this case.
\begin{thm}\label{thm:semiconcave}
The restriction of the value function, $u$,  to $H_0^1\times[0,T]$ is semiconcave.
\end{thm}
\begin{proof}
It will be shown that for every constant $B$, every closed interval $I
\subset [0,T)$, and all starting positions
$(\varphi_0^1,t^1)$ and $(\varphi_0^2,t^2)$ with
$||\varphi_0^1||_{H_0^1(0,1)} +||\varphi_0^2||_{H_0^1(0,1)} \leq B$
  and $t^1,t^2 \in I$, there is a constant C such that 
\begin{equation*}
u(\varphi^1_0,t^1)+u(\varphi^2_0,t^2)-2u\Big(\frac{\varphi^1_0+\varphi^2_0}{2},\frac{t^1+t^2}{2}\Big)
\leq C (||\varphi^1_0-\varphi^2_0||^2_{H_0^1}+|t^1-t^2|^2).
\end{equation*}
In order to
keep  constants simple we use that $u$ may be defined in $H_0^1(0,1)
\times (-\infty, T]$, so that we may set $t^1=h$ and $t^2=-h$, and realize that the
result for other times follows analogously. In this proof we let $C$
be any constant which may depend on $B$.

Let $\alpha:[0,T]\rightarrow L^2$ be an optimal control for the cost
functional $v_{\frac{\varphi^1_0+\varphi^2_0}{2},0}$ defined in
\eqref{eq:costfunctional}, and let $\varphi^3:[0,T]\rightarrow H_0^1$ be
the corresponding state. 
Define controls for solutions starting in $(\varphi^1_0,h)$ and
$(\varphi^2_0,-h)$ by dilations of $\alpha$ as
\begin{align*}
\alpha^1(t)&=\alpha\Big(T\frac{t-h}{T-h}\Big), \\
\alpha^2(t)&=\alpha\Big(T\frac{t+h}{T+h}\Big), \\
\end{align*}
and let the corresponding states be denoted $\varphi^1:[h,T]\rightarrow H_0^1$ with $\varphi^1(h)=\varphi^1_0$ and
$\varphi^2:[-h,T]\rightarrow H_0^1$ with $\varphi^2(-h)=\varphi^2_0$.
The evolution equation \eqref{eq:flow} for $\varphi^1$ and $\varphi^2$ is
now transformed to the interval $[0,T]$. The following equations are
thereby obtained:
\begin{align*}
\varphi^1_t &= \frac{T-h}{T}(\delta \varphi^1_{xx}-\delta^{-1} V'(\varphi^1)+\alpha), \quad
\varphi^1(0)=\varphi^1_0, \\
\varphi^2_t &= \frac{T+h}{T}(\delta\varphi^2_{xx}-\delta^{-1}V'(\varphi^2)+\alpha), \quad
\varphi^2(0)=\varphi^2_0, \\
\varphi^3_t &= \delta\varphi^3_{xx}-\delta^{-1}V'(\varphi^3)+\alpha, \quad
\varphi^3(0)=\frac{\varphi^1_0+\varphi^2_0}{2}. \\
\end{align*}

The function 
\begin{equation*}
z(t)=\varphi^1(t)+\varphi^2(t)-2\varphi^3(t)
\end{equation*}
is now introduced. We will obtain a bound for $||z(T)||$. 
The equation solved by $z$ is
\begin{multline}\label{eq:z}
z_t=\delta z_{xx}-\delta^{-1}\big(
V'(\varphi^1)+V'(\varphi^2)-2V'(\varphi^3)\big)\\
+\frac{\delta h}{T}(\varphi^2-\varphi^1)_{xx}
+\frac{\delta h}{T}\big( V'(\varphi^1)-V'(\varphi^2)\big).
\end{multline}
It is therefore necessary to find a bound for $\varphi^1-\varphi^2$. The
evolution equation for $\varphi^1-\varphi^2$ is
\begin{multline}\label{eq:phidiff}
(\varphi^1-\varphi^2)_t=\delta (\varphi^1-\varphi^2)_{xx}-\delta^{-1}\big(V'(\varphi^1)-V'(\varphi^2)\big)\\
-\frac{\delta h}{T}(\varphi^1+\varphi^2)_{xx}
  + \frac{\delta h}{T}\big(V'(\varphi^1)+V'(\varphi^2)\big) -\frac{2h}{T}\alpha.
\end{multline}
After the inner product is taken with $\varphi^1-\varphi^2$ the following
inequality is obtained:
\begin{multline*}
\half \frac{d}{dt}||\varphi^1-\varphi^2||^2 \leq \delta^{-1}|V''| \cdot
||\varphi^1-\varphi^2||^2 + \frac{\delta h}{T} ||\varphi^1_{xx}+\varphi^2_{xx}||\cdot
||\varphi^1-\varphi^2|| \\
+ \frac{\delta^{-1} h}{T} ||V'(\varphi^1)+V'(\varphi^2)|| \cdot
||\varphi^1-\varphi^2|| + \frac{2h}{T} ||\alpha||\cdot ||\varphi^1-\varphi^2||,
\end{multline*}
and hence
\begin{multline*}
\frac{d}{dt}||\varphi^1-\varphi^2|| \leq \delta^{-1}|V''| \cdot
||\varphi^1-\varphi^2|| + \frac{\delta h}{T}
||\varphi^1_{xx}+\varphi^2_{xx}|| \\
+ \frac{\delta^{-1} h}{T} ||V'(\varphi^1)+V'(\varphi^2)|| + \frac{2h}{T} ||\alpha||.
\end{multline*}
Thus, by Gr\"onwall's Lemma,
\begin{multline*}
||\varphi^1(t)-\varphi^2(t)|| \leq
  e^{\delta^{-1}|V''|T}||\varphi^1(0)-\varphi^2(0)||+\\
e^{\delta^{-1}|V''|T}\frac{h}{T} \int_0^T
  \big(\delta||\varphi^1_{xx} + \varphi^2_{xx}|| + \delta^{-1}||V'(\varphi^1)+V'(\varphi^2)|| + 2||\alpha||\big)dt.
\end{multline*}
Since $||\varphi_0^1||_{H_0^1(0,1)}+||\varphi_0^2||_{H_0^1(0,1)} \leq B$ it follows that 
$\varphi^1$ and $\varphi^2$ are bounded by a constant $C$ in
$L^2(0,T;H^2)$;
see \cite{Evans}. Together with the fact that $V'$ is bounded this
implies that 
\begin{equation}\label{eq:phidrift}
||\varphi^1(t)-\varphi^2(t)|| \leq C(||\varphi^1_0-\varphi^2_0|| + h), \quad
  \text{for all } 0 \leq t \leq T.
\end{equation}
Equation \eqref{eq:phidiff} is now used once again
together with the fact that $|V'(\varphi^1)-V'(\varphi^2)| \leq
|V''|\cdot|\varphi^1-\varphi^2|$ and Theorem 5 on page 360 in \cite{Evans}, to
draw the conclusion that
\begin{equation}\label{eq:H1difference}
\esssup_{0 \leq t \leq T}||\varphi^1(t)-\varphi^2(t)||_{H_0^1} + ||\varphi^1_{xx}-\varphi^2_{xx}||_{L^2(0,T;L^2)} \leq
  C(||\varphi^1_0-\varphi^2_0||_{H_0^1} + h).
\end{equation}
There is also the term $V'(\varphi^1)+V'(\varphi^2)-2V'(\varphi^3)$ in
\eqref{eq:z}. This can be handled as
\begin{multline}\label{eq:Vsplit}
|V'(\varphi^1)+V'(\varphi^2)-2V'(\varphi^3)|\\
 \leq
 |V'(\varphi^1)+V'(\varphi^2)-2V'(\frac{\varphi^1+\varphi^2}{2})| +
 2|V'(\frac{\varphi^1+\varphi^2}{2}) - V'(\varphi^3)| \\
\leq \frac{|V'''|}{2}
 |\varphi^1-\varphi^2|^2 + |V''|\cdot |z|.
\end{multline}
We are now ready to take the inner product with $z$ in \eqref{eq:z} to
obtain
\begin{multline*}
\half\frac{d}{dt} ||z||^2 \leq \frac{\delta^{-1}|V'''|}{2} \int_0^1
(\varphi^1-\varphi^2)^2|z|dx + \delta^{-1}|V''|\cdot||z||^2 \\
+
\frac{\delta h}{T}||\varphi^1_{xx}-\varphi^2_{xx}|| \cdot ||z|| +
|V''| \frac{\delta^{-1} h}{T}
||\varphi^1-\varphi^2|| \cdot ||z||,
\end{multline*}
which implies
\begin{multline*}
\frac{d}{dt} ||z|| \leq \delta^{-1}|V''|\cdot||z|| +
\frac{\delta^{-1}|V'''|}{2}||\varphi^1-\varphi^2||_{L^4(0,1)}^2 \\
+ \frac{\delta
  h}{T}||\varphi^1_{xx}-\varphi^2_{xx}||  + |V''| \frac{\delta^{-1} h}{T}
||\varphi^1-\varphi^2||.
\end{multline*}
By Gr\"onwall's Lemma
\begin{multline}\label{eq:zdrift}
||z(T)|| \leq e^{\delta^{-1}|V''|T} \int_0^T
\big(\frac{\delta^{-1}|V'''|}{2}||\varphi^1-\varphi^2||_{L^4(0,1)}^2 \\
+ \frac{\delta h}{T}||\varphi^1_{xx}-\varphi^2_{xx}||  + |V''| \frac{\delta^{-1}h}{T}
||\varphi^1-\varphi^2||)dt. 
\end{multline}
Sobolev's inequality gives that $||\varphi^1-\varphi^2||_{L^4(0,1)}
\leq C||\varphi^1-\varphi^2||_{H^1_0(0,1)}$, so \eqref{eq:H1difference}
together with \eqref{eq:zdrift}
implies that 
\begin{equation*}
||z(T)|| \leq C \big( ||\varphi^1_0-\varphi^2_0||^2_{H^1_0(0,1)} + h^2 \big).
\end{equation*}
This fact is now used to show that 
\begin{equation}\label{eq:vconcave}
v_{\varphi^1_0,h}(\alpha^1)+v_{\varphi^2_0,-h}(\alpha^2)-2v_{\frac{\varphi^1_0+\varphi^2_0}{2},0}(\alpha)
\leq C \big(||\varphi^1_0-\varphi^2_0||^2 + h^2 \big),
\end{equation}
The terminal cost is treated first. We use
the notation $\varphi_T \equiv \varphi(T)$ and perform a simple rearrangement:
\begin{multline}\label{eq:terminalcosts}
|g(\varphi^1_T)+g(\varphi^2_T)-2g(\varphi^3_T)| \\ 
= 
\big| \frac{K}{2}||\varphi^1_T -\varphi^2_T||^2 + K \big(\frac{\varphi^1_T+\varphi^2_T}{2}+
\varphi^3_T -2 \varphi_-, \varphi^1_T+\varphi^2_T-2\varphi^3_T\big) \big| \\
\leq C \big(||\varphi^1_0-\varphi^2_0||^2_{H^1_0} + h^2 \big),
\end{multline}
where \eqref{eq:phidrift}, \eqref{eq:zdrift}, and the fact that
$\varphi^1_T$, $\varphi^2_T$ and $\varphi^3_T$, are bounded are used. The
running costs must also be treated. A simple calculation shows that
\begin{equation}\label{eq:runningcosts}
\int_h^T ||\alpha^1||^2 dt + \int_{-h}^T ||\alpha^2||^2 dt -2 \int_0^T
||\alpha||^2 dt =0.
\end{equation}
The desired result \eqref{eq:vconcave} follows from
\eqref{eq:terminalcosts} and \eqref{eq:runningcosts}.
\end{proof}
\section{Discretization in space}\label{sec:spatialdiscr}
We shall compare the value functions associated with our original
problem and a finite element approximation. The value function we want
to approximate is
$u$ defined in \eqref{eq:valuefunction}.
%\begin{equation}\label{eq:exactu}
%u(\varphi_0,t)=\inf_{\alpha \in L^2(t,T;L^2(0,1))}\big( g(\varphi(T))+
%\int_t^T h(\alpha)ds \big), \quad \varphi(t)=\varphi_0,
%\end{equation}
%where $\varphi$ is the mild solution to 
%\begin{equation}\label{eq:phievol}
%\varphi_t=\varphi_{xx}-V'(\varphi)+\alpha, \quad \varphi(t,0)=\varphi(t,1)=0.
%\end{equation}
The approximate value function is, similarly as in \eqref{eq:valuefunction},
\begin{equation}\label{eq:approxu}
\bar u (\bar\varphi_0,t_0)=\inf_{\bar\alpha \in L^2(t_0,T;V)} \Big\{ g(\bar \varphi(T))+
\int_t^T h(\bar \alpha)ds\ :\   \bar\varphi(t_0)=\bar\varphi_0\Big\},
\end{equation}
where $\bar\varphi \in C(t_0,T;V)$ solves
\begin{equation}\label{eq:approxphievol}
(\bar\varphi_t,v)=-\delta (\bar\varphi_x,v_x)+(-\delta^{-1}
  V'(\bar\varphi)+\bar\alpha,v), \quad
  \text{for all } v \in V,
\end{equation}
and $V$ is the space of continuous piecewise linear functions on $[0,1]$ which
are zero at $0$ and $1$ and linear on the intervals $(0,\Delta x)$,
$(\Delta x, 2\Delta x)$, and so on. We note that the infima in
(\ref{eq:valuefunction}) and (\ref{eq:approxu}) are attained, using
Corollary \ref{cor:existence} for the original problem
\eqref{eq:valuefunction} and the easier theory in
\cite{Cannarsa-Sinestrari} for the approximation problem
\eqref{eq:approxu}. Therefore  we
can replace the infima with minima. 
The same sort of convergence analysis which is presented here is
performed for problems of optimal design in \cite{Carlsson-Sandberg-Szepessy}.

We now introduce some
notation needed in Theorem \ref{thm:valueerror}. We denote by
$P$ the $L^2$
projection from $L^2(0,1)$ to $V$ or from $L^2(0,1)\times\Re$ to $V
\times \Re$. Let $\Omega$ be an open subset of a Hilbert space $X$,
and $z:\Omega \rightarrow \Re$. For any $x_0 \in \Omega$ the
superdifferential $D^+ z(x_0)$ is defined as follows:
\begin{equation*}
D^+ z(x_0) = \Big\{ p \in X \big| \limsup_{x \rightarrow x_0}
  \frac{z(x)-z(x_0)-(p,x-x_0)_X}{|x-x_0|_X} \leq 0 \Big\}.
\end{equation*} 
The Hamiltonian, $H$, for the optimal
control problem \eqref{eq:valuefunction} is given by
\begin{equation}\label{eq:Hamiltonian}
H(\lambda,\varphi)=-\delta(\lambda_x,\varphi_x) -\delta^{-1}\big(\lambda,V'(\varphi)\big)-||\lambda||^2/2,
\end{equation} 
for all $\lambda$, $\varphi \in H_0^1(0,1)$. The restrictions of $u$ to
the subspaces $V \times [0,T]$ and $H_0^1 \times [0,T]$ are denoted
$u_V$ and $u_H$.
\begin{thm}\label{thm:valueerror}
Let $\varphi_0 \in V$. 
Denote an optimal pair
(control and state) for $u(\varphi_0,t_0)$ by $\alpha:[t_0,T]\rightarrow
L^2$ and $\varphi:[t_0,T]\rightarrow L^2$ and an optimal pair
for $\bar u(\varphi_0,t_0)$ by $\bar\alpha:[t_0,T]\rightarrow V$ and
$\bar\varphi:[t_0,T]\rightarrow V$.
Then
\begin{multline}\label{eq:errorrepr}
\int_{t_0}^T \Big(p^*_t(s)+H\big(P p^*_\varphi(s), \bar\varphi(s)\big)\Big)ds \\
\leq \bar u(\varphi_0,t_0)-u(\varphi_0,t_0) \\
\leq g\big(P \varphi(T)\big)-g\big(\varphi(T)\big)+
\int_{t_0}^T \Big(H\big(p^\#_\varphi(s),P \varphi(s)\big)-H\big(p^\#_\varphi(s),\varphi(s)\big)\Big)ds
\end{multline}
where $p^*(s)=\big(p^*_\varphi(s),p^*_t(s)\big) \in L^2(0,1)\times \Re$
is any measurable function with values in $D^+ u_H(\bar\varphi(s),s)$, and
$p^\#(s)=\big(p^\#_\varphi(s),p^\#_t(s)\big) \in V\times \Re$ is any
measurable function with values  in
$D^+ \bar u(P \varphi(s),s)$.
\end{thm}
\begin{proof}
We divide the proof into two steps: In \emph{Step 1} we obtain a lower bound for
$\bar u(\varphi_0,t_0)-u(\varphi_0,t_0)$, and in \emph{Step 2} we do
likewise for $u(\varphi_0,t_0)-\bar u(\varphi_0,t_0)$. 

\emph{Step 1.} \quad Using the definitions (\ref{eq:valuefunction}) and
(\ref{eq:approxu}) for $u$ and $\bar u$, the fact that
$\bar u \big(\bar\varphi(T),T\big)=g\big(\bar\varphi(T)\big)$, and that $u_H$ is the
restriction of $u$ to $H_0^1 \times [0,T]$, we can write 
\begin{equation}\label{eq:udiff1}
\begin{split}
\bar u(\varphi_0,t_0)-u(\varphi_0,t_0)&=u_H\big(\bar\varphi(T),T\big)-u_H\big(\bar\varphi(t_0),t_0\big)+\int_{t_0}^T
h(\bar\alpha)ds \\
&=\int_t^T \Big(\frac{d}{ds} u_H\big(\bar\varphi(s),s\big)+h\big(\bar\alpha(s)\big)\Big)ds,
\end{split}
\end{equation}
since $u_H(\bar\varphi(s),s)$ is absolutely continuous as $u$ is
locally Lipschitz continuous (see \cite{Cannarsa-Frankowska2}) and $\bar\varphi$ is absolutely continuous
as a function of $s$.

We now use that $u_H$ is a semiconcave function (with linear modulus),
so that for every $p \in D^+ u_H(z_0)$ there exists a constant $K$
such that
\begin{equation}\label{eq:semiconmodulus}
u_H(z)-u_H(z_0)- (p,z-z_0 ) \leq K |z-z_0|^2
\end{equation}
for all $z$ in a neighborhood of $z_0 \in H_0^1(0,1) \times (0,T)$; see \cite{Cannarsa}. Let now $p^*(s)=\big(p^*_\varphi(s),p^*_t(s)\big)$ be any element
in $D^+u_H\big(\bar\varphi(s),s\big) \cap \big(L^2(0,1)\times \Re\big)$.
Consider a point $s$ where the derivative
$\bar\varphi_t(s)$ exists. 
A lower bound for the backward derivative of $u_H\big(\bar\varphi(s),s\big)$ will
now be obtained. We split the difference quotient approximating the
backward derivative at $s$:
\begin{align*}
&\frac{u_H(\bar\varphi(s),s)-u_H(\bar\varphi(s-\Delta s),s-\Delta s)}{\Delta s}
\\
=& -\big[ u_H\big(\bar\varphi(s-\Delta s),s-\Delta
  s\big)-u_H\big(\bar\varphi(s),s\big)\\
&-p^*_t(s)(-\Delta s) -
  \big(p^*_\varphi(s),\bar\varphi(s-\Delta s)-\bar\varphi(s)\big)\big]/{\Delta s}\\
& +p^*_t(s) + \Big(p^*_\varphi(s),\frac{\bar\varphi(s)-\bar\varphi(s-\Delta s)}{\Delta
s}\Big).
\end{align*}
If equation \eqref{eq:semiconmodulus} is used together with the
fact that $\bar\varphi$ is differentiable at $s$ it can be deduced that
the quotient involving the square bracket in the above equation is greater than or equal to $-K'
\Delta s$, for some constant $K'$. Letting $\Delta s \rightarrow 0$ we
see that 
\begin{equation*}
\frac{d}{ds}u_H \big(\bar\varphi(s),s\big) \geq p^*_t(s) +\big(p^*_\varphi(s),\bar\varphi_t(s)\big),
\end{equation*} 
where (temporarily) $d/ds$ denotes the backward derivative.
In order to be able to apply (\ref{eq:approxphievol}) we
note that $\bar\varphi_t \in V$ implies
\begin{equation*}
(p^*_\varphi,\bar\varphi_t)=(P p^*_\varphi,\bar\varphi_t).
\end{equation*}
Thus the integrand in
(\ref{eq:udiff1}), using the backward derivative,  can be bounded from
below as follows:
\begin{multline*}
\frac{d}{ds} u_H\big(\bar\varphi(s),s\big)+h\big(\bar\alpha(s)\big) 
\geq p^*_t(s)+\big(\bar\varphi_t(s),P p^*_\varphi(s)\big)+h\big(\bar\alpha(s)\big) \\
= p^*_t(s)-\delta\big(\bar\varphi_{x}(s),(P
p^*_\varphi(s))_x\big)-\delta^{-1}\Big(V'\big(\bar\varphi(s)\big),P p^*(s)\Big) \\
+\big(\bar\alpha(s), P p^*_\varphi(s)\big) +
\half ||\bar\alpha(s)||^2 \\
\geq p^*_t(s)+ H\big(P p^*_\varphi(s),\bar\varphi(s)\big),
\end{multline*}
since
\begin{equation*}
H(\lambda,\varphi)=\min_{\alpha \in L^2(0,1)} \Big(
-\delta(\varphi_x,\lambda_x)
-\delta^{-1}\big(V'(\varphi),\lambda\big)+(\alpha,\lambda) + \half ||\alpha||^2\Big).
\end{equation*}
The double sided and the
backward time derivatives of $u_H(\bar\varphi(s),s)$ differ on a set
of measure zero, so there is no problem in using the backward
derivative in (\ref{eq:udiff1}).

\emph{Step 2.} Lower bound for $u(\varphi_0,t_0)-\bar u(\varphi_0,t_0)$. It is
now assumed that $\varphi_0 \in V$. Similarly as in \emph{Step 1} we
write, noting that $\bar u$ is only defined on $V\times [0,T]$,
\begin{equation}\label{eq:udiffg}
\begin{split}
& u(\varphi_0,t_0)-\bar u(\varphi_0,t_0) \\
&= g\big(\varphi(T)\big)-g\big(P \varphi(T)\big) +
\bar u\big(P \varphi(T),T\big)-\bar u\big(P \varphi(t_0),t_0\big) +
\int_{t_0}^T h\big(\alpha(s)\big)ds \\
&=g\big(\varphi(T)\big)-g\big(P \varphi(T)\big) +\int_{t_0}^T \Big(\frac{d}{ds} \bar
u\big(P \varphi(s),s\big) +h\big(\alpha(s)\big)\Big)ds =: I + II.
\end{split}
\end{equation}
A lower bound for part $II$ is  obtained by splitting the
difference quotient approximating the backward derivative at $s$: 
\begin{equation}\label{eq:bdsplit}
\begin{split}
&\frac{\bar u\big(P \varphi(s),s\big)-\bar u\big(P \varphi(s-\Delta s),s-\Delta s\big)}{\Delta s}
\\
=& -\big[\bar u\big(P \varphi(s-\Delta s),s-\Delta
  s\big)-\bar u\big(P \varphi(s),s\big)\\
&-p^\#_t(s)(-\Delta s) -
  \big(p^\#_\varphi(s),P \varphi(s-\Delta s)-P \varphi(s)\big)\big]/{\Delta s}\\
& +p^\#_t(s) + \Big(p^\#_\varphi(s),\frac{P \varphi(s)-P \varphi(s-\Delta s)}{\Delta
s}\Big).
\end{split}
\end{equation}
The derivative $\varphi_t(s)$ exists for $t_0 < s <T$ by the theory in
e.g.\ Chapter 3 in \cite{Henry}, where we have used also that the
control, $\alpha=-\lambda$, solves an adjoint backward parabolic PDE,
and therefore is H\"older continuous. It is now used that $||P x||
\leq ||x||$, $\bar u$ is semiconcave (see e.g.\ \cite{Cannarsa-Sinestrari}), and that
\begin{equation*}
\Big(p^\#_\varphi(s),\frac{P \varphi(s)-P \varphi(s-\Delta s)}{\Delta
s}\Big) = \Big(p^\#_\varphi(s),\frac{\varphi(s)-\varphi(s-\Delta s)}{\Delta
s}\Big)
\end{equation*}
in  equation \eqref{eq:bdsplit}, so that we have, similarly as in \emph{Step 1}, that 
\begin{equation*}
\frac{d}{ds}\bar u (P \varphi(s),s) \geq p^\#_t(s) +(p^\#_\varphi(s),\bar\varphi_t(s)).
\end{equation*}
By further using Chapter 3 in \cite{Henry} it is known that equation
\eqref{eq:flow} is satisfied in the $L^2$ sense, with $\varphi(s) \in
H^2(0,1) \cap H_0^1(0,1)$, for $t_0 < s <T$. Similarly as in
\emph{Step 1}, the integrand in \eqref{eq:udiffg}, using the backward
derivative can be bounded from below:
\begin{equation*}
\frac{d}{ds} \bar u(P \varphi(s),s) + h(\alpha(s)) \geq 
p^\#_t(s) + H(p^\#_\varphi(s),\varphi(s)).
\end{equation*}
As $\bar u$ is a viscosity solution to the Hamilton-Jacobi equation
for the discrete value function it holds that
\begin{equation*}
p^\#_t(s) + H(p^\#_\varphi(s),P \varphi(s)) \geq 0,
\end{equation*}
which proves the second inequality in \eqref{eq:errorrepr}.
%
%The semiconcavity of
%$\bar u$ implies that the directional derivative to $\bar u$ exists in
%every direction at every point. As in \emph{Step 1} we therefore have
%\begin{multline*}
%\frac{\bar u(\text{Proj}_V \varphi(s+\Delta s),s+\Delta
%  s)-\bar u(\text{Proj}_V \bar\varphi(s),s)}{\Delta s} \\
%= \partial
%  \bar u(\text{Proj}_V \varphi(s),s;( \frac{\text{Proj}_V \varphi(s+\Delta
%  s)-\text{Proj}_V \varphi(s)}{\Delta s},
%  1)) + o(1)\\
%=\min_{p \in D^+ \bar u} \big(p_t +(p_\varphi,\frac{\text{Proj}_V \varphi(s+\Delta
%  s)- \text{Proj}_V \varphi(s)}{\Delta s})\big) + o(1) \\
%=\min_{p \in D^+ \bar u} \big(p_t +(p_\varphi,\frac{\varphi(s+\Delta
%  s)- \varphi(s)}{\Delta s})\big) + o(1).
%\end{multline*}
%It can now be used that 
%\begin{equation*}
%\varphi \in L^2(t,T;H^2(0,1)) \cap C(t,T;H^1_0(0,1))
%\end{equation*}
%so that 
%\begin{align*}
%&(p_\varphi,\frac{\varphi(s+\Delta
%  s)- \varphi(s)}{\Delta s}) \\
%&= \frac{1}{\Delta s}\int_s^{s+\Delta s}\big(
%  -(\partial_x p_\varphi,\varphi_x)+(p_\varphi,-V'(\varphi)+\alpha)\big)ds
%  \rightarrow -(\partial_x p_\varphi,\varphi_x(t))+(p_\varphi,-V'(\varphi(t))+\alpha(t))
%\end{align*}
%We thus see that
%\begin{equation*}
%\frac{d}{ds} \bar
%u(\text{Proj}_V \varphi(s),s) +h(\alpha(s) = p_t +H(p^\#_\varphi(s),\varphi(s))
%\geq -H(p^\#_\varphi(s),\text{Proj}_V \varphi(s)) + H(p^\#_\varphi(s),\varphi(s))
%\end{equation*} 
%which gives us the desired lower bound.
\end{proof}
Theorem \ref{thm:valueerror} will be used when the error between the
original and the approximate value functions is computed. For this to
work some knowledge about the superdifferential $D^+ u_H$ is
needed. The dual equation
\begin{subequations}\label{eq:l}
\begin{align}
-\lambda_t &= \delta \lambda_{xx} - \delta^{-1}\lambda V''(\varphi), \label{eq:l1}\\
 \lambda(T)&=2K \big( \varphi(T) - \varphi^-\big), \label{eq:l2}
\end{align}
\end{subequations}
is introduced. Let $\alpha$ and $\varphi$ be optimal pairs as in Theorem
\ref{thm:valueerror}. According to Theorem \ref{thm:boundedcontrol} it
is possible to choose a bounded control. For the mild solution $\lambda$  to \eqref{eq:l} 
there exists, according to Theorem 3.1 in \cite{Cannarsa-Frankowska1}, a
subset $\mathcal{L} \subset [t_0, T]$, of full measure, such that, for
all $t \in \mathcal{L}$,
\begin{equation}\label{eq:valuesinDplus}
\varphi(t) \in H_0^1(0,1) \cap H^2(0,1) \implies
\Big( \lambda(t), -H\big(\lambda(t), \varphi(t)\big) \Big) \in D^+ u \big(\varphi(t),t\big).
\end{equation} 
By the same theorem, it holds that for almost every $t \in [t_0,T]$,
\begin{equation}\label{eq:lambdaalpha}
\big( \lambda(t),\alpha(t)\big) + \frac{||\alpha||^2}{2} =
\min_{\substack{a \in L^2(0,1) \\ ||a|| \leq L}} \Big(
  \big(\lambda(t),a \big) + \frac{||a||^2}{2} \Big),
\end{equation}
where $L$ is the bound on the control from Theorem
\ref{thm:boundedcontrol}. This bound is included in order to be able
to use the aforementioned Theorem 3.1 in \cite{Cannarsa-Frankowska1},
but since $\alpha$ and $\lambda$ correspond to the original problem
\eqref{eq:valuefunction} with no bound we could also have used any
constant greater than $L$ in \eqref{eq:lambdaalpha}. Hence we see that
\eqref{eq:lambdaalpha} holds also without the requirement that $a$ is
bounded. From this we draw the conclusion that $\lambda(t) =
-\alpha(t)$ a.e. The mild solutions $\varphi$ and $\lambda$ therefore
satisfy the system
\begin{subequations}\label{eq:mildsols}
\begin{align}
\varphi(t) &= S(t-t_0)\varphi_0 + \int_{t_0}^t S(t-s) \big(-\delta^{-1}
V'(\varphi(s)) -\lambda(s) \big)ds, \label{eq:mildsols1} \\
\varphi(t_0) &=\varphi_0, \label{eq:mildsols2} \\
\lambda(t) &= S(T-t)\lambda(T) - \delta^{-1} \int_{t}^T S(s-t) \big(
\lambda(s)V''(\varphi(s)) \big)ds, \label{eq:mildsols3} \\
\lambda(T) &= 2K(\varphi(T)-\varphi^-), \label{eq:mildsols4}
\end{align}
\end{subequations}
where $S(t)$ is the contraction semigroup of linear operators
generated by $\delta \frac{d^2}{dx^2}$, \eqref{eq:mildsols1} is
equation \eqref{eq:phimild} with $\lambda = -\alpha$, and
\eqref{eq:mildsols3} is the equation for mild solutions to
\eqref{eq:l1}; see e.g.\ Theorem 3.1 in \cite{Cannarsa-Frankowska1}. 
Following the notation in \cite{Henry} we introduce
\begin{equation*}
A \equiv -\delta \frac{d^2}{dx^2}.
\end{equation*}
The operator $A$ has eigenvalues $k_n = \delta \pi^2 n^2$, $(n=1,2,3,
\ldots )$ with corresponding eigenfunctions $\psi_n(x) = \sqrt{2} \sin
(n\pi x)$. Fractional powers of $A$ may be defined using this:
\begin{equation}\label{eq:Agamma}
A^\gamma \varphi = \sum_{n=1}^\infty k_n^\gamma(\psi_n,\varphi)\psi_n,
\end{equation}
for $\gamma \geq 0$. The domain for $A^\gamma$ is given by
\begin{equation}\label{eq:Agammadomain}
D(A^\gamma) = \Big\{ \varphi \in L^2(0,1): \ \sum_{n=1}^\infty k_n^{2\gamma}
(\psi_n,\varphi)^2 < \infty \Big\}.
\end{equation}
For $\gamma=1$ and $\gamma=1/2$ we have that $||A \varphi|| = \delta ||\varphi_{xx}||$ and
$||A^{1/2}\varphi|| = \sqrt{\delta} ||\varphi_x||$.
We state a few useful properties of the fractional powers of $A$,
which may be found in e.g.\ \cite{Henry}.
For any $K>0$ and all $0 < \gamma < K$ there exists a constant $C$
such that 
\begin{subequations}\label{eq:Aproperties}
\begin{align}
||A^\gamma S(t)|| &\leq C t^{-\gamma}, \quad \text{for $t>0$,}
   \label{eq:Aproperties1}\\
\intertext{and if $0 < \gamma \leq 1, \varphi \in D(A^\gamma)$,}
||(S(t)-I)\varphi|| & \leq \frac{1}{\gamma}C t^\gamma ||A^\gamma
\varphi||. \label{eq:Aproperties2} \\
\intertext{It also holds that}
A^{\gamma^1}A^{\gamma^2}&=A^{\gamma^2}A^{\gamma^1}=A^{\gamma^1+\gamma^2}
\ \text{ on $D(A^{\gamma^1+\gamma^2})$ when $\gamma^1, \gamma^2 \geq
    0$,} \label{eq:Aproperties3} \\
A^\gamma S(t)&=S(t) A^\gamma \ \text{ on $D(A^\gamma)$, $t>0$.} \label{eq:Aproperties4}
\end{align}
\end{subequations}

In  the following Theorem it is shown how an element in $D^+
u_H(\bar\varphi(s),s)$ can be obtained, which is needed according to Theorem \ref{thm:valueerror}.
\begin{thm}\label{thm:HJsatisfied}
When $\varphi_0 \in V$ and $0 \leq s \leq T$ one element in $D^+
u_H(\varphi_0,s)$ is given by
\begin{equation*}
\Big( \lambda(s), -H\big(\lambda(s),\varphi_0\big) \Big),
\end{equation*}
where $\lambda$ is a mild solution to  \eqref{eq:l} 
and $\varphi$ is an optimal solution to \eqref{eq:valuefunction} with
$t_0=s$.
%We may therefore take 
%\begin{equation*}
%\int_{t_0}^T \big( -H(\lambda(s),\bar\varphi(s))+H(P
%\lambda(s),\bar\varphi(s)) \big)ds
%\end{equation*}
%as the first integral in \eqref{eq:errorrepr}. 
\end{thm} 
\begin{proof}
%According to Theorem 3.1 in \cite{Cannarsa-Frankowska1} there exists a
%subset $\mathcal{L} \subset [t_0, T]$, of full measure, such that, for
%all $t \in \mathcal{L}$,
%\begin{equation*}
%\varphi(t) \in H_0^1(0,1) \cap H^2(0,1) \implies
%\big( \lambda(t), -H(\lambda(t), \varphi(t)) \big) \in D^+ u (\varphi(t),t).
%\end{equation*} 
Using Chapter 3 in \cite{Henry} we have that 
$\varphi(t) \in H_0^1(0,1) \cap H^2(0,1)$ for $t_0 < t < T$, so by
\eqref{eq:valuesinDplus} 
\begin{equation*}
\big( \lambda(t), -H(\lambda(t), \varphi(t)) \big) \in D^+ u (\varphi(t),t),
\quad \text{a.e.}
\end{equation*} 
It follows from the definition that then also
\begin{equation*}
\big( \lambda(t), -H(\lambda(t), \varphi(t)) \big) \in D^+ u_H (\varphi(t),t),
\quad \text{a.e.}
\end{equation*}
We shall verify that the semiconcavity of $u_H$ implies that
\begin{equation}\label{eq:uppersemicontinuous}
z_n \rightarrow z_0,\ D^+u_H(z_n) \ni p_n \rightarrow p \implies 
p \in D^+ u_H(z_0).
\end{equation}
In order to prove this we use \eqref{eq:semiconmodulus} at the points
$z_n$. We thereby have that 
\begin{equation*}
\begin{split}
&u_H(z) - u_H(z_0) - (p,z-z_0)\\ 
= & u_H(z)-u_H(z_n) -(p_n,z-z_n) \\
& \quad+  u_H(z_n) - u_H(z_0) + (p_n-p,z-z_0) +
(p_n,z - z_n) \\
 \leq & K |z-z_0|^2 + \varepsilon,
\end{split}
\end{equation*}
where $\varepsilon$ can be made arbitrarily small by using the
convergence $z_n \rightarrow z_0$, $p_n \rightarrow p$, and that $u_H$
is continuous (it is even locally Lipschitz continuous, see
\cite{Cannarsa-Frankowska2}). Hence
\begin{equation*}
u_H(z) - u_H(z_0) - (p,z-z_0) \leq  K |z-z_0|^2,
\end{equation*}
which implies that $p \in D^+u_H(z_0)$, and
\eqref{eq:uppersemicontinuous} holds. 
(As can be seen in the above argument it suffices that $p_n
\rightharpoonup p$ weakly,
but we will not need this here.)

Since the Hamiltonian $H:H_0^1(0,1) \times H_0^1(0,1) \rightarrow \Re$
is locally Lipschitz continuous and \eqref{eq:valuesinDplus} holds,
what remains is to prove that $\varphi$ and $\lambda$ are continuous as
functions of time with values in $H_0^1(0,1)$ at
$t_0$. By equation \eqref{eq:mildsols1} we
have that
\begin{multline}\label{eq:Ahalfevol}
A^{1/2} \big(\varphi(t)-\varphi_0 \big) = \big( S(t-t_0) -I \big)
A^{1/2}\varphi_0 \\
+ \int_{t_0}^t A^{1/2} S(t-s) \big(
-\delta^{-1}V'(\varphi(s)) -\lambda(s)\big)ds,
\end{multline} 
where passing $A^{1/2}$ under the integral sign is justified by the
fact that $A^{1/2}$ is a closed operator. By \eqref{eq:Agammadomain}
it is a straightforward calculation to confirm that $V \subset D(A^\gamma)$ for
$\gamma < 3/4$. Since $\varphi_0 \in V$, \eqref{eq:Aproperties2}  may be
used to get a bound for the first term in the right hand side of
\eqref{eq:Ahalfevol}:
\begin{equation*}
||\big(S(t-t_0)-I \big) A^{1/2}\varphi_0 || \leq 10 C t^{1/10} ||A^{3/5} \varphi_0||.
\end{equation*} 
The norm of the integral in \eqref{eq:Ahalfevol} converges to zero as
$t \rightarrow t_0$ by \eqref{eq:Aproperties1} and the fact that $V'$
and $\lambda$ (since it equals $-\alpha$) are bounded. Hence 
\begin{equation*}
||A^{1/2} \big(\varphi(t) - \varphi_0 \big)|| \rightarrow 0, \quad \text{as
  $t \searrow t_0$.}
\end{equation*} 
The function $\lambda$ is also   continuous as a function with values in $H_0^1(0,1)$
when $t \searrow t_0$, as, by Theorem 3.5.2 in \cite{Henry},
$||A^\gamma \lambda_t||$ exists when $t<T$ and $\gamma <1$, e.g.\ $\gamma=1/2$.
\end{proof}
In order to be able to use Theorem \ref{thm:valueerror} and
\ref{thm:HJsatisfied} a few results about the regularity for the state
and the dual is established. The original setting, without
discretization in space, is considered first.
%\begin{lemma}\label{lem:projection}
%\begin{equation*}
%D\arrowvert_V^+ u(t,x)=\text{Proj}_V D^+ u(t,x), \quad \text{when } x
%\in V.
%\end{equation*}
%\end{lemma}
%\begin{proof}
%It is a direct consequence of the definition that $\text{Proj}_V D^+
%u(t,x) \subset D\arrowvert_V^+ u(t,x)$. Hence we show that 
%$D\arrowvert_V^+ u(t,x)\subset \text{Proj}_V D^+ u(t,x)$. This is done by
%proving that $D\arrowvert_V^* u(t,x)\subset \text{Proj}_V D^* u(t,x)$. Let
%\begin{equation*}
%p=\lim_{i \to \infty} D\arrowvert_V u (t_i,x_i) \in D\arrowvert_V^* u(t,x),
%\end{equation*}
%where $x_i \in V$ for all $i$. Since $D^+ u$ is upper semicontinuous
%there exists points $(s_i,y_i)$ where $y_i \in L^2$ and where u is
%differentiable such that 
%\begin{equation}\label{eq:uppersemicont}
%||(t_i,x_i)-(s_i,y_i)|| < \frac{1}{i}, \quad Du(s_i,y_i) \in
%  D^+u(t_i,x_i) + \frac{1}{i}B.
%\end{equation}
%The fact that 
%\begin{equation*}
%\text{Proj}_V Du(s_i,y_i)= D\arrowvert_V u(s_i,y_i)
%\end{equation*}
%together with (\ref{eq:uppersemicont}) thus entails
%\begin{equation}\label{eq:projconvergence}
%\lim_{i \to \infty} \text{Proj}_V Du(s_i,y_i) = p.
%\end{equation}
%The set $\{Du(s_i,y_i)\}$ is bounded in $L^2$, so there exists a
%subsequence, for which the subscript $i$ is also used, which converges
%weakly in $L^2$ to  an element $q \in D^* u(t,x)$. This fact together
%with (\ref{eq:projconvergence}) gives us that 
%\begin{equation*}
%\text{Proj}_V q=p.
%\end{equation*} 
%\end{proof}
\begin{thm}\label{thm:infinitedimregularity}
For every $C>0$ and all starting positions $(\varphi_0,t_0)$ satisfying
$||\varphi_0||_{H_0^1(0,1)} \leq C$, $0 \leq t_0 \leq T$, there exists a $D>0$ and an optimal
state $\varphi$ to problem \eqref{eq:valuefunction}, with corresponding
dual $\lambda$, solving \eqref{eq:l}, such that for all $t_0 \leq t
\leq T$,
\begin{subequations}\label{eq:infinitedimregularity}
\begin{align}
||\varphi(t)||_{H_0^1(0,1)} & \leq D, \label{eq:infinitedimregularity1}
  \\
||\varphi(t)||_{H^2(0,1)} & \leq D(t-t_0)^{-1/2}, \label{eq:infinitedimregularity2}
  \\
||\lambda(t)||_{H_0^1(0,1)} & \leq D, \label{eq:infinitedimregularity3}
  \\
||\lambda(t)||_{H^2(0,1)} & \leq
  D(T-t_0)^{-1/2}. \label{eq:infinitedimregularity4} \\
\intertext{If $\varphi_0$ satisfies the higher regularity $||A^{5/7}
  \varphi_0|| \leq C$ it further holds that}
||\varphi(t)||_{H^2(0,1)} &\leq D(t-t_0)^{-2/7},  \label{eq:infinitedimregularity5} \\
||\lambda(t)||_{H^2(0,1)} &\leq D(T-t_0)^{-2/7}.  \label{eq:infinitedimregularity6} 
\end{align}
\end{subequations}
\end{thm}
\begin{proof}
In the proof, we will write $D$ for any constant which may depend on
$C$, but not on $t_0$.

\emph{Step 1.} By theorem \ref{thm:boundedcontrol} it is possible to
choose an optimal control $\alpha$ such that $||\alpha(t)|| \leq L$,
for some constant $L$ which only depends on $C$. Since
$\lambda=-\alpha$ the same holds for $\lambda$.

\emph{Step 2.} Since $A^\gamma$ and $S(t)$ commute (see
\eqref{eq:Aproperties4}) we can operate with $A^{1/2}$ on equation
\eqref{eq:mildsols1} to obtain
\begin{equation*}
A^{1/2}\varphi(t)=S(t-t_0)A^{1/2}\varphi_0 + \int_{t_0}^t A^{1/2} S(t-s)
\big(-\delta^{-1} V'(\varphi(s))-\lambda(s) \big)ds.
\end{equation*} 
As $S(t)$ is a contraction semigroup and by the boundedness of $V'$
and $||\lambda(s)||$ together with \eqref{eq:Aproperties1} it therefore holds that 
\begin{equation*}
||A^{1/2}\varphi(t)|| \leq ||A^{1/2}\varphi_0|| + D \int_{t_0}^t
  (t-s)^{-1/2} ds,
\end{equation*}
and hence
$||\varphi_x(t)|| \leq D$. By a Poincar\'e
inequality (e.g.\ Proposition 5.3.5 in \cite{Brenner-Scott})
\eqref{eq:infinitedimregularity1} holds.

\emph{Step 3.} Since $\lambda(T)=2K(\varphi(T)-\varphi^-)$, boundedness of
$\lambda(T)$ in $H_0^1(0,1)$ follows. The same analysis for
\eqref{eq:mildsols3}  as was performed in \emph{Step 2} may therefore
be used. Using that $||\lambda(s)||$ is bounded for all $s$ gives
\eqref{eq:infinitedimregularity3}.

\emph{Step 4.}  Operate with $A$ on \eqref{eq:mildsols1} to obtain
\begin{equation}\label{eq:Ahalfhalf}
A\varphi(t) = A^{1/2}S(t-t_0) A^{1/2}\varphi_0 + \int_{t_0}^t A^{1/2} S(t-s)
A^{1/2} \big(-\delta^{-1}V'(\varphi(s)) -\lambda(s)\big) ds.
\end{equation}
Since $\varphi(s)$ and $\lambda(s)$ are bounded in $H_0^1(0,1)$ for all
$s$ it holds that
\begin{equation*}
||A^{1/2} \big(-\delta^{-1}V'(\varphi(s))-\lambda(s)\big)|| < D.
\end{equation*}
Therefore 
\begin{equation*}
||A \varphi(t)|| \leq D (t-t_0)^{-1/2} + D \int_{t_0}^t (t-s)^{-1/2} ds \leq
  D (t-t_0)^{-1/2},
\end{equation*}
as $T$ is finite. So \eqref{eq:infinitedimregularity2} holds.

\emph{Step 5.} In this last step we use the operator $A$ in equation
\eqref{eq:mildsols3} in the following way:
\begin{equation}\label{eq:H2lambda}
A\lambda(t) = S(T-t) A\lambda(T) - \delta^{-1} \int_t^T A^{1/2} S(s-t)
A^{1/2} \Big(\lambda(s) V''\big(\varphi(s)\big)\Big)ds.
\end{equation}
The bound for $\varphi(T)$ in $H^2(0,1)$ may be transferred to
$\lambda(T)$ by \eqref{eq:mildsols4}, which yields $||\lambda(T)||
\leq D(T-t_0)^{-1/2}$. As both $\varphi(t)$ and $\lambda(t)$ are bounded in $H_0^1(0,1)$
for all $t_0 < t < T$ the integral in \eqref{eq:H2lambda} is bounded
by 
\begin{equation*}
D \int_t^T (s-t)^{-1/2} ds.
\end{equation*}  
Since $T$ is finite
\eqref{eq:infinitedimregularity4} holds.

\emph{Step 6.} Equations \eqref{eq:infinitedimregularity5} and \eqref{eq:infinitedimregularity6} can be
proved similarly as in \emph{Step 4} and \emph{Step 5} by changing the
first term in the right hand side of \eqref{eq:Ahalfhalf} to 
$A^{2/7}S(t-t_0)A^{5/7} \varphi_0$. By this we see that $||A\varphi(t)||
\leq D(t-t_0)^{-2/7}$, which implies \eqref{eq:infinitedimregularity6},
just as in \emph{Step 5}.
\end{proof}
A regularity result for the spatially discretized case is now to be
established.
According to theory in e.g.\ \cite{Cannarsa-Sinestrari} the optimal
control problem \eqref{eq:approxu}, \eqref{eq:approxphievol} has a
minimizing control $\bar \alpha$.  The corresponding state is denoted
$\bar \varphi$. The value function
is differentiable along optimal paths, i.e.\ the derivative exists at $\bar u (\bar\varphi(s),s)$
for all $s \in (t_0,T]$. The spatial G\^ateaux derivative of $\bar u$ at
$(\bar\varphi(s),s)$ will be denoted $\bar\lambda(s)$. The optimal state,
$\bar\varphi$, and the G\^ateaux derivative $\bar\lambda$ satisfy the
following system:
\begin{subequations}\label{eq:approxHamiltonsyst}
\begin{align}
(\bar\varphi_t,v)&=-\delta(\bar\varphi_{x},v_x)-(\delta^{-1}V'(\bar\varphi)+\bar\lambda,v), \quad
  \text{for all } v \in V,\label{eq:approxHamiltonsyst1}\\
\bar\varphi(t_0)&=\bar\varphi_0, \label{eq:approxHamiltonsyst2}\\
-(\bar\lambda_t,v)&=
  -\delta(\bar\lambda_x,v_x)-\delta^{-1}(\bar\lambda V''(\bar\varphi),v),
  \quad \text{for all } v \in V, \label{eq:approxHamiltonsyst3}\\
\bar\lambda(T) &= 2K(\bar\varphi(T)-\text{Proj}_V \varphi^-). \label{eq:approxHamiltonsyst4}
\end{align}
\end{subequations}
Furthermore, the theory in \cite{Cannarsa-Sinestrari} reveals that the
optimal control, $\bar\alpha$, satisfies $\bar\alpha=-\bar\lambda$.
Some new notation is now introduced. We let the interval $[0,1]$ be
divided into $M$ subintervals with $\Delta x = 1/M$, and let
$\{v^i\}_{i=1}^{M-1}$ be the standard nodal basis in $V$; see Figure
\ref{fig:basis}.
\begin{figure}
\centering
\includegraphics[width=0.5\textwidth]{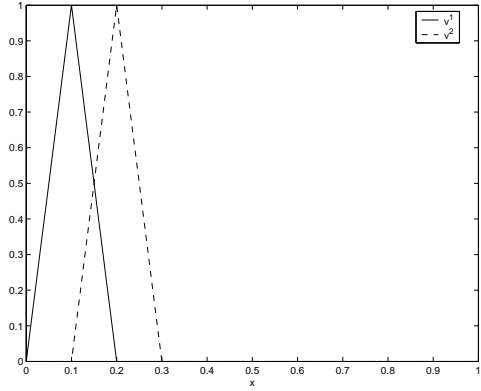}
\caption{Two of the nine basis functions when $\Delta x=1/10$.}
\label{fig:basis}
\end{figure} 
The \emph{interpolant}, $I$, takes a function in $C([0,1])$ to the
element in $V$ which
coincides with the original function at $i \Delta x$, $i=1, \ldots ,
M-1$, so that, for instance, $I \lambda(i \Delta x) = \lambda( i
\Delta x)$.
The second difference quotient matrix $D^2$ and the mass matrix $B$ are
introduced as
\begin{align}
D^2 &= \frac{1}{\Delta x^2} \begin{pmatrix}
-2 & 1 & 0 & \cdots & 0 \\
1 & -2 & 1 &  & \vdots & \\
0 & \ddots &\ddots &\ddots & 0 \\
\vdots & & 1 & -2 & 1 \\
0 & \cdots & 0 & 1 & -2
\end{pmatrix}, \label{eq:diffquotient}\\
\notag \\
B &= \begin{pmatrix}
2/3 & 1/6 & 0 & \cdots & 0 \\
1/6 & 2/3 & 1/6 &  & \vdots & \\
0 & \ddots &\ddots &\ddots & 0 \\
\vdots & & 1/6 & 2/3 & 1/6 \\
0 & \cdots & 0 & 1/6 & 2/3
\end{pmatrix}. \label{eq:massmatrix}
\end{align}
If $\bar\varphi(t)$ and $\bar\lambda(t)$ are written in the basis
$\{v^i\}_{i=1}^{M-1}$ as
\begin{equation}\label{eq:coorexpr}
\bar\varphi(t) =: \sum_{i=1}^{M-1} \zeta^i(t) v^i, \quad
\bar\lambda(t) =: \sum_{i=1}^{M-1} \theta^i(t) v^i,
\end{equation}
then equations \eqref{eq:approxHamiltonsyst1} and
\eqref{eq:approxHamiltonsyst3} may be rewritten as 
\begin{subequations}\label{eq:FEMcoords}
\begin{align}
B \zeta' &= \delta D^2 \zeta - \frac{\delta^{-1}}{\Delta x} p - B
\theta, \label{eq:FEMcoords1} \\
-B \theta' &= \delta D^2 \theta - \frac{\delta^{-1}}{\Delta x}r, \label{eq:FEMcoords2}
\end{align}
\end{subequations}
where 
\begin{equation}\label{eq:vectdefs}
\begin{split}
\zeta= \begin{pmatrix} \zeta^1 \\ \vdots \\ \zeta^{M-1} \end{pmatrix},&
\quad 
\theta= \begin{pmatrix} \theta^1 \\ \vdots \\ \theta^{M-1}
\end{pmatrix}, \\ 
p=\begin{pmatrix}
(V'(\bar\varphi),v^1) \\
\vdots \\
(V'(\bar\varphi),v^{M-1})
\end{pmatrix},
&\ \text{and} \   
r=\begin{pmatrix}
(\bar\lambda V''(\bar\varphi),v^1) \\
\vdots \\
(\bar\lambda V''(\bar\varphi),v^{M-1})
\end{pmatrix}.
\end{split}
\end{equation}
We now state a Lemma which will be used in the proofs of Theorem
\ref{thm:diffcontrol} and \ref{thm:spatialconvergence}.
\begin{lemma}\label{lem:H2proj}
For any element $\psi \in H^2(0,1) \cap H^1_0(0,1)$ the projection
$P\psi$, written in the nodal basis $\{v^i\}$ as
\begin{equation*}
P \psi =: \sum_{i=1}^{M-1} \xi^i v^i,
\end{equation*}
satisfies
\begin{equation}\label{eq:H2proj}
\big(\Delta x \sum_{i=1}^{M-1} (D^2 \xi)_i^2 \big)^{1/2} \leq C ||\psi_{xx}||,
\end{equation}
with a constant $C$ independent of $\Delta x$.
\end{lemma} 
\begin{proof}
The vector $\xi$ is defined as
\begin{equation*}
\xi= \begin{pmatrix} \xi^1 \\ \vdots \\ \xi^{M-1} \end{pmatrix},
\end{equation*}
and is given by $\xi=\frac{1}{\Delta x} B^{-1} q$, where 
\begin{equation*}
q=\begin{pmatrix}
(\psi,v^1) \\
\vdots \\
(\psi,v^{M-1})
\end{pmatrix},
\end{equation*}
which follows from the fact that $(\psi,v^i)=(P\psi,v^i)$ for all $v^i$.
The matrices $D^2$ and $B^{-1}$ commute, since the eigenvectors of
$D^2$ and $B$ (and $B^{-1}$) are equal, so $D^2 \xi = \frac{1}{\Delta
x} B^{-1} D^2 q.$ Every element of the vector $D^2
q$, except the first and the last, is a $L^2(0,1)$ scalar product
between $\psi$ and a translate of the function $v$ in Figure \ref{fig:v}.
\begin{figure}
\centering
\includegraphics[width=0.5\textwidth]{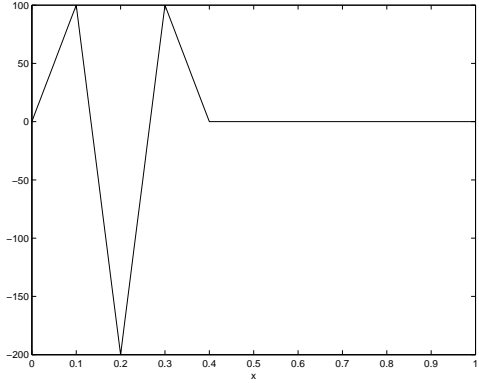}
\caption{The function $v$ admits values in the interval $[-2/\Delta
    x^2, 1/\Delta x^2]$ (here $\Delta x = 1/10$).}
\label{fig:v}
\end{figure}
It is easy to check that there exists a primitive function $\bar v$ to $v$
such that $\bar v(0)=\bar v(4\Delta x)=0$, and that there exists a primitive
function $\Bar{\Bar{v}}$ to $\bar v$ such that $\Bar{\Bar{v}}(0)=\Bar{\Bar{v}}(4\Delta x)=0$. Furthermore $\max
\Bar{\Bar{v}} \leq E$, for a constant $E$ which does not depend on $\Delta x$. Hence,
\begin{equation*}
|(\psi,v)|=|(\psi_{xx},\Bar{\Bar{v}})| \leq E \int_0^{4\Delta x} |\psi_{xx}|dx
\leq 2E \sqrt{\Delta x} \big( \int_0^{4\Delta x} \psi_{xx}^2 dx \big)^{1/2}.
\end{equation*}
The same sort of bound may be obtained
also for the first and the last elements of $D^2 q$
by using a $2$-periodic, odd extension of $\psi$ outside
$[0,1]$. We therefore have that
\begin{align*}
||D^2 q||_2^2 & \leq 4E^2 \Delta x \big(\underbrace{\int_{0}^{\Delta x} \psi_{xx}^2 dx
  + \int_0^{3\Delta x} \psi_{xx}^2 dx}_{\text{From } (D^2 q)_1} + \underbrace{\int_0^{4\Delta x} \psi_{xx}^2 dx}_{\text{From } (D^2 q)_2}
  + \ldots \\
&+ \underbrace{\int_{(M-3)\Delta x}^1 \psi_{xx}^2 dx + \int_{(M-1)\Delta
  x}^1\psi_{xx}^2 dx}_{\text{From } (D^2 q)_{M-1}} \big) \\
& \leq 16 E^2 \Delta x ||\psi_{xx}||^2,
\end{align*}
where $||\cdot||_2$ denotes the standard Euclidean vector norm. The
eigenvalues of $B^{-1}$ lie in the interval $[1,3]$, and hence it
holds that 
\begin{equation*}
||D^2 \xi||_2 \leq \frac{12E}{\sqrt{\Delta x}} ||\psi_{xx}||,
\end{equation*}
which is equivalent to \eqref{eq:H2proj}. 
\end{proof}
\begin{thm}\label{thm:diffcontrol}
There are constants $E$ and $F$, depending on the parameters of the
optimal control problem, as in Theorem \ref{thm:boundedcontrol}, but
not on $\varphi_0$, $t_0 \in [0,T]$ and the size of the
spatial discretization, such that for all $t_0 \leq t \leq T$,
\begin{align*}
||\bar\varphi_x(t)|| + ||\bar\lambda_x(t)|| & \leq E||(\varphi_0)_x|| + F, \\
||\bar\varphi_t(t)|| &\leq E\big( \Delta x
\sum_{i=1}^{M-1}(D^2 \zeta(t_0))_i^2 \big)^{1/2} +F, \\
\intertext{and}
||\bar\lambda_t(t)|| & \leq \Big(E\big( \Delta x
\sum_{i=1}^{M-1}(D^2 \zeta(t_0))_i^2 \big)^{1/2} +F \Big)^2, \\
\intertext{where}
\varphi_0 &= \sum_{i=1}^{M-1} \zeta^i(t_0) v^i. 
\end{align*}  
%For each $C>0$, there is a $D>0$, independent of the size of the
%spatial discretization, $\Delta x$, and starting time, $t_0$, such that for all
%$||(\bar\varphi_0)_x||\leq C$ 
%there exists an optimal state $\bar\varphi : [t_0,T] \rightarrow V$ for
%\eqref{eq:approxu}, with an associated dual $\bar\lambda:[t_0,T]
%\rightarrow V$ which satisfies \eqref{eq:approxHamiltonsyst}, such
%that for all $t_0 \leq t \leq T$,
%\begin{equation*}
%||\bar\varphi_x(t)|| + ||\bar\lambda_x(t)|| \leq D.
%\end{equation*}  
%If, moreover
%\begin{equation*}
%\varphi_0 = \sum_{i=1}^{M-1} \zeta^i(t_0) v^i
%\end{equation*}
%satisfies
%\begin{equation}\label{eq:initcondbound}
%\big( \Delta x \sum_{i=1}^{M-1}(D^2 \zeta(t_0))_i^2 \big)^{1/2} \leq C,
%\end{equation}
%then also $||\bar\varphi_t(t)||+||\bar\lambda_t(t)|| \leq D$, for all $t_0 \leq t \leq T$.
\end{thm}
\begin{proof}
The proof uses the same kind of techniques as in the proof of Theorem
5 on page 360 in \cite{Evans}.
Here, however, the regularity of $\bar\varphi$ must also be conveyed to $\bar\lambda$.
In the proof, we will write $E$ and $F$ for any constants that may depend on
the parameters of the problem, but not on $\Delta x$,  $t_0$, and $\varphi_0$.

\emph{Step 1.} As in the infinite dimensional case, treated in Theorem
\ref{thm:boundedcontrol}, it holds that the infimum in
\eqref{eq:approxu}  can be changed to
$\inf_{||\bar\alpha||_{L^{\infty}(t_0,T;V)}\leq E||\varphi_0||+F}$. The proof goes just as the proof for Theorem
  \ref{thm:boundedcontrol}. By a Poincar\'e inequality (see e.g.\ Theorem 5.3.5 in
\cite{Brenner-Scott}) it holds that $||\varphi_0|| \leq C ||(\varphi_0)_x||$
and therefore the infimum can be written  
$\inf_{||\bar\alpha||_{L^{\infty}(t_0,T;V)}\leq E||(\varphi_0)_x||+F}$.
It is therefore possible to let $v=\bar\varphi$ in
\eqref{eq:approxphievol} and use this boundedness of
$\bar\alpha$ to see that 
$||\bar\varphi(t)|| \leq E ||(\varphi_0)_x|| +F$, for all $t_0 \leq t \leq T$.

\emph{Step 2.} As already noted $\bar\lambda=-\bar\alpha$, and so by
\emph{Step 1} the same bound on $||\bar\lambda(t)||$ also holds.

\emph{Step 3.} With
$v=\bar\varphi_t$ in \eqref{eq:approxHamiltonsyst1} we have
\begin{multline*}
||\bar\varphi_t ||^2 + \frac{\delta}{2} \frac{d}{dt} ||\bar\varphi_x||^2 = -
  \delta^{-1} (V'(\bar\varphi),\bar\varphi_t)-(\bar\lambda,\bar\varphi_t) \\
\leq
  \delta^{-1} |V'|\cdot||\bar\varphi_t||+ ||\bar\lambda||\cdot
  ||\bar\varphi_t|| 
\leq \delta^{-2}|V'|^2+\frac{||\bar\varphi_t||^2}{4} +
  \frac{||\bar\lambda||^2}{2} + \frac{||\bar\varphi_t||^2}{2}.
\end{multline*}
The boundedness of $(\bar\varphi_0)_x$ and $||\bar\lambda(t)||$ thus implies
that 
\begin{align*}
||\bar\varphi_x(t)|| &\leq E||(\varphi_0)_x||+F, \\
||\bar\varphi_t||_{L^2(t_0,T;V)} &\leq E||(\varphi_0)_x||+F.
\end{align*}
By this the same sort of bound holds for  $\bar\lambda_x(T)$.
Letting $v=\bar\lambda_t$ in \eqref{eq:approxHamiltonsyst3} gives
that 
\begin{align*}
||\bar\lambda_x(t)|| &\leq E||(\varphi_0)_x||+F, \\
||\bar\lambda_t||_{L^2(t_0,T;V)} &\leq E||(\varphi_0)_x||+F,
\end{align*}
similarly as for $\bar\varphi$.

\emph{Step 4.}  All eigenvalues
of $B$ lie in the interval $[1/3,1]$, and so $||B^{-1}||_2 \leq 3$
(independently of $\Delta x$), where $||\cdot||_2$ denotes the
operator 2-norm. From this and \eqref{eq:FEMcoords} it follows
that 
\begin{equation*}
||\bar\varphi_t(t_0)|| \leq E\big( \Delta x
\sum_{i=1}^{M-1}(D^2 \zeta(t_0))_i^2 \big)^{1/2} +F.
\end{equation*}
We now introduce the notation $\hat\varphi \equiv
\bar\varphi_t$ and $\hat \lambda \equiv \bar\lambda_t$ and differentiate
equation \eqref{eq:approxHamiltonsyst1}
with respect to time:
\begin{equation*}
(\hat\varphi_t,v)=-\delta(\hat\varphi_{x},v_x)-\delta^{-1}(V''(\bar\varphi)\hat\varphi,v)-(\hat\lambda,v),
  \quad \text{for all } v \in V.
\end{equation*}
Let $v=\hat\varphi$ to obtain
\begin{multline*}
\half \frac{d}{dt} ||\hat\varphi||^2 = -\delta ||\hat
\varphi_x||^2-\delta^{-1} (V''(\bar\varphi)\hat\varphi,\hat\varphi)-
(\hat\lambda,\hat\varphi) \\
\leq \delta^{-1} |V''| \cdot ||\hat\varphi||^2 +
\frac{||\hat\lambda||^2}{2} + \frac{||\hat\varphi||^2}{2}.
\end{multline*}
The fact that $||\bar\varphi_t(t_0)|| \leq D$
together with the result on boundedness of
$||\hat\lambda||_{L^2(t_0,T;V)}$ and $||\hat\varphi||_{L^2(t_0,T;V)}$ in \emph{Step 3} implies 
that 
\begin{equation*}
||\hat\varphi(t)|| \leq E\big( \Delta x
\sum_{i=1}^{M-1}(D^2 \zeta(t_0))_i^2 \big)^{1/2} +F, \quad \text{for }
t_0 \leq t \leq T.
\end{equation*}

\emph{Step 5.} In this step it will  be shown that
$| \big((P\varphi_-)_x,w_x\big)| \leq D ||w||$ for all $w \in V$. Some new notation is introduced:
\begin{equation*}
P\varphi_- \equiv \sum_{i=1}^{M-1} \xi^i v^i, \quad w \equiv
\sum_{i=1}^{M-1} \eta^i v^i,
\end{equation*}  
with corresponding  vectors $\xi$ and $\eta$.
By means of a partial integration
\begin{equation}\label{eq:phiprod}
|\big( (P\varphi_-)_x,w_x\big)| \leq \Delta x |\eta^{T} D^2 \xi| \leq
\Delta x ||\eta||_2\cdot||D^2 \xi||_2,
\end{equation}
with $||\cdot||_2$ denoting the Euclidean vector norm.
Since the eigenvalues of $B$ lie in $[1/3,1]$ it holds that 
\begin{equation*}
||w||^2 = \Delta x \eta^T B \eta \geq \frac{\Delta x}{3} ||\eta||^2_2,
\end{equation*}
and so by Lemma \ref{lem:H2proj}, \eqref{eq:phiprod} gives:
\begin{equation*}
|\big( (P\varphi_-)_x,w_x \big)| \leq \Delta x ||D^2 \xi||_2 \cdot
 ||\eta||_2 \leq \Delta x \frac{D||(\varphi_-)_{xx}||}{\sqrt{\Delta x}}
 \cdot \frac{\sqrt{3}}{\sqrt{\Delta x}} ||w|| = D||w||.
\end{equation*}

\emph{Step 6.} It holds by equation \eqref{eq:approxHamiltonsyst4}
that $\bar\lambda_x(T)=2K(\bar\varphi_x(T)-(P\varphi_-)_x)$. Using this in
\eqref{eq:approxHamiltonsyst3} as well as
\eqref{eq:approxHamiltonsyst1} gives
\begin{align*}
-(\hat\lambda(T),v) &=
 -2K\delta(\bar\varphi_x(T),v_x)+2K\delta((P\varphi_-)_x,v_x)-\delta^{-1}(\bar\lambda(T)V''(\bar\varphi(T)),v) \\
&=
 2K(\hat\varphi(T),v)+2K(\delta^{-1}V'(\bar\varphi(T))+\bar\lambda(T),v) \\
&\quad +2K\delta((P\varphi_-)_x,v_x) -\delta^{-1}(\bar\lambda(T)V''(\bar\varphi(T)),v).
\end{align*}
With $v=\hat\lambda(T)$ it follows that 
\begin{equation*}
||\hat\lambda(T)|| \leq E\big( \Delta x
\sum_{i=1}^{M-1}(D^2 \zeta(t_0))_i^2 \big)^{1/2} +F,
\end{equation*}
by the results in \emph{Step 4} and
\emph{Step 5}. In order to bound 
$\hat\lambda$ at all
times, equation
\eqref{eq:approxHamiltonsyst3} is differentiated with respect to time:
\begin{equation*}
-(\hat\lambda_t,v)=-\delta(\hat\lambda_x,v_x)-\delta^{-1}(\hat\lambda
 V''(\bar\varphi),v)-\delta^{-1}(\bar\lambda V'''(\bar\varphi)\hat\varphi,v).
\end{equation*}
With $v=\hat\lambda$ in the previous equation the following bound is
obtained:
\begin{equation*}
-\half\frac{d}{dt}||\hat\lambda||^2 \leq
 \delta^{-1}|V''|\cdot||\hat\lambda||^2 + \delta^{-1}|V'''| \int_0^1
 |\bar\lambda \hat\varphi \hat\lambda|dx.
\end{equation*}
Since $\bar\lambda$ is bounded by $E\big( \Delta x
\sum_{i=1}^{M-1}(D^2 \zeta(t_0))_i^2 \big)^{1/2} +F$ in $H^1_0$ for all times 
it is similarly bounded in $L^\infty$. The last integral in the previous inequality may therefore be
estimated as follows:
\begin{equation*}
\int_0^1
 |\bar\lambda \hat\varphi \hat\lambda|dx \leq |\bar\lambda| \cdot
 ||\hat\varphi|| \cdot ||\hat\lambda|| \leq \Big(E\big( \Delta x
\sum_{i=1}^{M-1}(D^2 \zeta(t_0))_i^2 \big)^{1/2} +F\Big)^2 ||\hat\lambda||.
\end{equation*}
Using $\frac{d}{dt}||\hat\lambda||^2 =
2||\hat\lambda||\frac{d}{dt}||\hat\lambda||$, Gr\"onwall's Lemma, and
the boundedness of $||\hat\lambda(T)||$ we see that $||\hat\lambda(t)||$ is bounded by $\Big(E\big( \Delta x
\sum_{i=1}^{M-1}(D^2 \zeta(t_0))_i^2 \big)^{1/2} +F\Big)^2$ for all times.
\end{proof}
With the error representation in Theorem \ref{thm:valueerror},
Theorem \ref{thm:HJsatisfied} about $D^+ u_H$, and the regularity results of Theorem \ref{thm:infinitedimregularity}
and \ref{thm:diffcontrol}, it is possible to prove Theorem
\ref{thm:spatialconvergence} about spatial convergence.
We first give the idea of the proof. When the first integral in
\eqref{eq:errorrepr} is estimated the optimal path $\bar\varphi(s)$ is
used. In order to obtain an element in $D^+
u_H\big(\bar\varphi(s),s\big)$ the system of equations
\eqref{eq:mildsols} is considered with $\big(\bar\varphi(s),s\big)$
playing the role of $(\varphi_0,t_0)$. By Theorem~
\ref{thm:HJsatisfied} the computed $\lambda(s)$ is the spatial part of
an element in $D^+ u_H\big(\bar\varphi(s),s\big)$. Therefore, equation
\eqref{eq:mildsols} must be used for every starting position
$\big(\bar\varphi(s),s\big)$, $0 \leq s \leq T$, as depicted by Figure
\ref{fig:manyphis}. 
\begin{figure}
\centering
\includegraphics[width=0.5\textwidth]{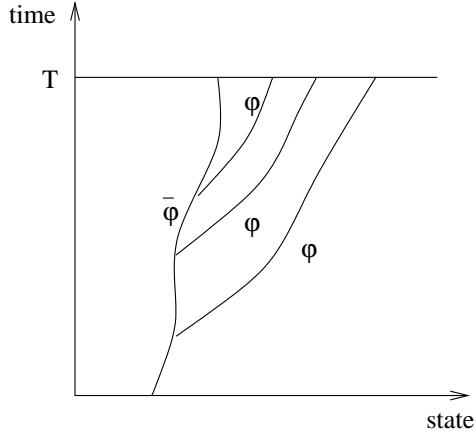}
\caption{Sketch showing the idea in Theorem \ref{thm:spatialconvergence}.}
\label{fig:manyphis}
\end{figure}
Similarly, when the second integral in \eqref{eq:errorrepr} is
estimated, the optimal path $\varphi(s)$ is used, and for each
$\big(P \varphi(s),s\big)$ as starting positions the solution
$\bar\lambda(s)$ to \eqref{eq:approxHamiltonsyst} is computed. 
\begin{thm}\label{thm:spatialconvergence}
For every constant $C>0$ and all starting points $\varphi_0$ with
\begin{equation}\label{eq:startingposs}
\varphi_0 = \sum_{i=1}^{M-1} \xi^i v^i \in V
\end{equation}
such that
\begin{equation}\label{eq:startingposcond}
\big( \Delta x \sum_{i=1}^{M-1} (D^2 \xi)_i^2 \big)^{1/2} \leq C,
\end{equation}
there is a constant $D>0$ such that
\begin{equation}\label{eq:deltaxconv}
|u(\varphi_0,0) - \bar u (\varphi_0,0)| \leq D \Delta x^2.
\end{equation}
\end{thm}
\begin{remark}
Every reasonable approximation in $V$ of $\varphi_+$, for any $\Delta x$,
satisfies \eqref{eq:startingposcond}. The interpolant and the
projection are possible choices.
\end{remark}
\begin{proof}
As in the  proof of Theorem \ref{thm:diffcontrol}, whenever $D$ is
written in this proof it means a constant independent of $\Delta x$, but (possibly) dependent on $C$.
  
\emph{Step 1.} It will be shown that condition \eqref{eq:startingposcond} implies
\begin{equation}\label{eq:Afivefourths}
||A^{5/7} \varphi_0 || \leq D.
\end{equation}
Using \eqref{eq:Agamma} it follows that
\begin{equation}\label{eq:Asum}
||A^\gamma \varphi || =\big( \sum_{n=1}^\infty k_n^{2\gamma} (\psi_n,\varphi)^2 \big)^{1/2},
\end{equation}
where $k_n$ and $\psi_n$ are the eigenvalues and eigenfunctions of
$A$. Two partial integrations imply that
\begin{equation*}
(\psi_n,\varphi_0)= -\frac{1}{n^2 \pi^2} (\psi_n, (\varphi_0)_{xx}),
\end{equation*}
where $(\varphi_0)_{xx}$ is the distributional second derivative of
$\varphi_0$, i.e.\ 
\begin{equation*}
(\varphi_0)_{xx} = \Delta x \sum_{i=1}^{M-1} (D^2 \xi)_i \bar\delta_{i\Delta x},
\end{equation*}
with $\bar\delta_{i\Delta x}$ the Dirac delta distribution in $x=i \Delta
x$. Since $|\psi_n(x)| \leq \sqrt{2}$ for all $n \in \NatNum$ and all
$x \in [0,1]$ it holds that
\begin{align*}
|(\psi_n,\varphi_0)| & \leq \frac{\sqrt{2}}{n^2 \pi^2} \Delta x
 \sum_{i=1}^{M-1} |(D^2 \xi)_i| 
\leq \frac{\sqrt{2}}{n^2 \pi^2} \Delta x \sqrt{M-1} \big(
 \sum_{i=1}^{M-1} (D^2 \xi)_i^2 \big)^{1/2} \\ 
&\leq \frac{\sqrt{2}}{n^2
 \pi^2} \big( \Delta x \sum_{i=1}^{M-1} (D^2 \xi)_i^2 \big)^{1/2}
\leq \frac{\sqrt{2} C}{n^2 \pi^2}.
\end{align*}
It thereby follows that the sum in \eqref{eq:Asum} is finite (and
proportional to $C$) when $\gamma < 3/4$, e.g.\ $\gamma=5/7$.

\emph{Step 2.} In this step it is shown that there exists a solution
$\lambda$ to \eqref{eq:mildsols} with $\bar\varphi(s)$ playing the role
of $\varphi_0$, such that
\begin{equation}\label{eq:Hdiff1}
|H\big(P\lambda(s),\bar\varphi(s)\big)-H\big(\lambda(s),\bar\varphi(s)\big)| \leq  D \Delta
 x^2 (T-s)^{-1/2}.
\end{equation}
We start by showing that the starting position $\varphi_0$ in
\eqref{eq:startingposs} is bounded in $H_0^1(0,1)$:
\begin{multline*}
||(\varphi_0)_x||^2 = -\big( (\varphi_0)_{xx},\varphi_0 \big) \\
\leq D \big(\Delta x \sum_{i=1}^{M-1} (D^2 \xi)_i^2 \big)^{1/2}
||\varphi_0||_{L^\infty(0,1)} \leq D ||(\varphi_0)_x||, 
\end{multline*}
where the last inequality follows by a Sobolev inequality. Hence\linebreak $||(\varphi_0)_x||~ \leq~ D$.
By Theorem \ref{thm:diffcontrol} it follows that
$||\bar\varphi_x(s)|| \leq D$, for all $t_0 \leq s \leq T$. By Theorem \ref{thm:infinitedimregularity} it then
follows that $||\lambda(s)||_{H^2(0,1)} \leq D (T-s)^{-1/2}$.    
The Hamiltonian, $H$, consists of three parts; see
\eqref{eq:Hamiltonian}. The first of these is the most difficult when
\eqref{eq:Hdiff1} is to be proved, so we will focus on this one and 
let the other two parts be treated by the reader. The difference
between the first parts of the Hamiltonians in \eqref{eq:Hdiff1} is
given by
\begin{equation*}
-\Big(\bar\varphi_x(s),\big(\lambda(s)-P\lambda(s)\big)_x\Big) = 
\big(\bar\varphi_{xx}(s), \lambda(s)- P\lambda(s) \big),
\end{equation*}
where the factor $\delta$ is left out for convenience. We reuse
notation and let 
\begin{equation*}
\bar\varphi(s) =: \sum_{i=1}^{M-1} \xi^i v^i,
\end{equation*}
so that
\begin{equation*}
\bar\varphi_{xx}(s) = \Delta x \sum_{i=1}^{M-1} (D^2 \xi)_i
\bar\delta_{i\Delta x}.
\end{equation*}
As $||\bar\varphi_t(s)|| \leq D$ by Theorem \ref{thm:diffcontrol} it,
again using \eqref{eq:FEMcoords},
holds that 
\begin{equation*}
\big(\Delta x \sum_{i=1}^{M-1}(D^2 \xi)_i^2 \big)^{1/2} \leq D.
\end{equation*} 
Hence it holds that
\begin{multline}\label{eq:beforeinterpolant}
|\big(\bar\varphi_{xx}(s),\lambda(s)-P\lambda(s)\big)| \leq \Delta x 
\sum_{i=1}^{M-1} |(D^2 \xi)_i \big(\lambda(i\Delta x)- P\lambda(i\Delta
x)\big)| \\
\leq \big(\Delta x \sum_{i=1}^{M-1}(D^2 \xi)_i^2 \big)^{1/2}
\Big(\Delta x \sum_{i=1}^{M-1} \big(\lambda(i\Delta x) - P \lambda(i\Delta
x)\big)^2 \Big)^{1/2} \\
\leq D \Big(\Delta x \sum_{i=1}^{M-1} \big(\lambda(i\Delta x) - P \lambda(i\Delta
x)\big)^2 \Big)^{1/2}.
\end{multline}
By the use of the interpolant the last parenthesis in
\eqref{eq:beforeinterpolant} may be written
\begin{equation*}
\Big(\Delta x \sum_{i=1}^{M-1} \big(I \lambda(i\Delta x) - P \lambda(i\Delta
x)\big)^2 \Big)^{1/2}.
\end{equation*}
This minor difference simplifies the situation as $I \lambda(s) -P
\lambda(s) \in V$, which  makes comparison with the $L^2$ norm
possible. For an element 
\begin{equation*}
\kappa=\sum_{i=1}^{M-1} \eta^i v^i \in V
\end{equation*}
it holds that 
\begin{equation*}
||\kappa||^2 = \Delta x (\eta,B\eta)_2,
\end{equation*}
where $(\cdot,\cdot)_2$ is the Euclidean scalar product on $\Re^d$. All
eigenvalues of $B$ lie in the interval $[1/3,1]$, and hence
\begin{equation*}
||I\lambda -P\lambda||^2 \geq \third \Delta x \sum_{i=1}^{M-1} (\lambda(i\Delta x) - P \lambda(i\Delta
x))^2. 
\end{equation*}
It also holds that
\begin{equation*}
||I\lambda - P\lambda|| = ||P(I \lambda - \lambda)|| \leq ||I \lambda
  - \lambda|| \leq D \Delta x^2 ||\lambda||_{H^2(0,1)},
\end{equation*}
where the last inequality may be found in e.g.\ \cite{Brenner-Scott}.
By equation (\ref{eq:infinitedimregularity4}) we have that 
\begin{equation*}
||\lambda(s)||_{H^2(0,1)} \leq D (T-s)^{-1/2},
\end{equation*}
and so \eqref{eq:Hdiff1} holds. By Theorem \ref{thm:HJsatisfied}, an
element in $D^+ u_H\big(\bar\varphi(s),s\big)$ is given by
$\Big(\lambda(s),-H\big(\lambda(s),\bar\varphi(s)\big)\Big)$, and
thereby it is clear that the first
integral in \eqref{eq:errorrepr}
may be bounded by 
\begin{equation*}
D\Delta x^2\int_{t_0}^T (T-s)^{-1/2}ds \leq D\Delta x^2.
\end{equation*}  

\emph{Step 3.} In this step a bound for the second integral in
\eqref{eq:errorrepr} is derived. In \emph{Step 1} it was proved that 
$||A^{5/7} \varphi_0|| \leq D$. Theorem \ref{thm:infinitedimregularity}
then implies that $||\varphi(s)||_{H^2(0,1)} \leq
D(s-t_0)^{-2/7}$. Therefore, by Lemma \ref{lem:H2proj} and  Theorem \ref{thm:diffcontrol} there
exists a solution $\bar\lambda$ to \eqref{eq:approxHamiltonsyst}, with
$\big(P\varphi(s),s \big)$ in the role of $(\varphi_0,t_0)$, such that
\begin{equation*}
||\bar\lambda_t(s)|| \leq \big( D(s-t_0)^{-2/7} +F \big)^2 \leq D(s-t_0)^{-4/7}.
\end{equation*}
By \eqref{eq:FEMcoords2} it holds that
\begin{equation*}
\Big( \Delta x \sum_{i=1}^{M-1} \big(D^2 \theta(s)\big)_i^2
\Big)^{1/2} \leq D(s-t_0)^{-4/7},
\end{equation*}  
where $\theta$ is given by \eqref{eq:coorexpr} and
\eqref{eq:vectdefs}.
In order to be able to use the above information to get a bound of the
second integral in \eqref{eq:errorrepr}, we need that $\bar\lambda(s)$
is the spatial part of an element in $D^+ \bar u(P\varphi(s),s)$. This
follows from Lemma 3.3.16 and Theorem 7.4.17 in
\cite{Cannarsa-Sinestrari}. As in \emph{Step 2} we are satisfied with
considering only the first parts of the Hamiltonians. The only
difference is that now the partial integration is performed so that
$\bar\lambda(s)$ is distributionally differentiated twice:
\begin{multline*}
\Big| \Big(\bar\lambda_x(s),\big(\varphi(s)-P\varphi(s)\big)_x \Big) \Big| =
\big| \big(\bar\lambda_{xx}(s),\varphi(s)-P\varphi(s)\big)\big| \\
\leq \Big( \Delta x \sum_{i=1}^{M-1} \big(D^2 \theta(s)\big)_i^2
\Big)^{1/2} \Big(\Delta x \sum_{i=1}^{M-1}\big(\varphi(i\Delta
 x,s)-P\varphi(i\Delta x,s) \big)^2 \Big)^{1/2} \\
\leq D(s-t_0)^{-4/7} \Delta x^2 ||\varphi_{xx}(s)|| \leq D \Delta x^2 (s-t_0)^{6/7},
\end{multline*}
similarly as in \emph{Step 2}. The second integral in
\eqref{eq:errorrepr} is  therefore bounded by a term
\begin{equation*}
D \Delta x^2 \int_{t_0}^T (s-t_0)^{-6/7}ds = D \Delta x^2.
\end{equation*} 

\emph{Step 4.} It remains to show that the difference 
$g(P\varphi(T))-g(\varphi(T))$ is of the order $\Delta x^2$. Since
$||\varphi(T)||$ is uniformly bounded for all starting positions in a
bounded set in $L^2(0,1)$ we see by \eqref{eq:gdifference} that the difference in final
costs is less than $D||P\varphi(T)-\varphi(T)||$. Since $P I \varphi=I\varphi$,
where $I$ is the interpolant, introduced in \emph{Step2}, we have that
\begin{multline*}
||P\varphi(T)-\varphi(T)|| \leq ||P\big(\varphi(t)-I\varphi(T)\big)|| +
  ||I\varphi(T)-\varphi(T)|| \\
\leq 2 ||I\varphi(T)-\varphi(T)|| \leq D \Delta x^2
  ||\varphi_{xx}(T)|| \leq D \Delta x^2,
\end{multline*}
where the last inequality follows by  Theorem \ref{thm:infinitedimregularity}. 
\end{proof}
The next theorem provides an error estimate which makes comparison
with the situation where $\varphi_+$ is used as initial position
possible.
\begin{thm}\label{thm:comparisonphi+}
There exists a constant $D>0$ such that 
\begin{equation}\label{eq:comparisonphi+}
|u(\varphi_+,0)-u(P\varphi_+,0)| + |u(\varphi_+,0)-u(I\varphi_+,0)|
 \leq D \Delta x^2.
\end{equation}
\end{thm}
\begin{remark}
The theorem shows that both the projection and the interpolant can be
chosen when approximating $\varphi_+$ in $V$.
\end{remark}
\begin{proof}
The semiconcavity of $u_H$ implies that for every bounded set $X
\subset H_0^1(0,1)$ there exists a constant $D$, such that 
\begin{equation}\label{eq:beginningconcave}
\varphi^1, \varphi^2 \in X \implies u_H(\varphi^1,0) -u_H(\varphi^2,0)
- (p, \varphi^1-\varphi^2) \leq D||\varphi^1 - \varphi^2||^2_{H_0^1(0,1)},
\end{equation}
where $p$ is the spatial part of any element in $D^+ u_H(\varphi^2,0)$
(compare \eqref{eq:semiconmodulus}). In Theorem \ref{thm:HJsatisfied}
it was proved that $\Big(
\lambda(0),-H\big(\lambda(0),\varphi^1\big)\Big)$ is one such element,
where $\lambda$ is a solution to \eqref{eq:mildsols} with
$\varphi_0=\varphi^2$. We may therefore take $p=\lambda(0)$ in
\eqref{eq:beginningconcave}. By Theorem
\ref{thm:infinitedimregularity}, $||p|| \leq D$ for some constant $D$
(it is even bounded in $H_0^1(0,1)$, but this is not needed here).
Plugging this boundedness into \eqref{eq:beginningconcave} results in
the inequality
\begin{equation*}
u_H(\varphi^1,0)-u_H(\varphi^2,0) \leq D\big(||\varphi^1 - \varphi^2||
+ ||\varphi^1-\varphi^2||^2_{H_0^1(0,1)}\big).
\end{equation*}
We may change places for $\varphi^1$ and $\varphi^2$ everywhere above,
and thereby obtain
\begin{equation}\label{eq:beginningsemiconres}
|u_H(\varphi^1,0)-u_H(\varphi^2,0)| \leq D\big(||\varphi^1 - \varphi^2||
+ ||\varphi^1-\varphi^2||^2_{H_0^1(0,1)}\big).
\end{equation}
Consider now $\varphi^1 = \varphi_+$ and $\varphi^2 = I\varphi_+$ or
$\varphi^2 = P \varphi_+$. For the interpolant, $I$, the following
bounds hold:
\begin{equation}\label{eq:interpbounds}
\begin{split}
||\varphi_+ - I \varphi_+|| &\leq
  D||(\varphi_+)_{xx}||_{L^\infty(0,1)} \Delta x^2, \\
||\varphi_+ - I \varphi_+||_{H_0^1(0,1)} &\leq
  D||(\varphi_+)_{xx}|| \Delta x,
\end{split}
\end{equation}
with a constant $D$ independent of $\Delta x$ and $\varphi_+$. Since
$(\varphi_+)_{xx}$ is bounded in $L^\infty(0,1)$ this together with
\eqref{eq:beginningsemiconres} directly shows that the interpolant
part of \eqref{eq:comparisonphi+} is correct. The projection part is
proved by using the result from \cite{Crouzeix-Thomee}, that the $L^2$
projection is stable in $H_0^1(0,1)$, i.e.
\begin{equation*}
||P \varphi||_{H_0^1(0,1)} \leq D ||\varphi||_{H_0^1(0,1)}.
\end{equation*}
It therefore holds that 
\begin{multline}\label{eq:stableproj}
||\varphi_+ - P\varphi_+||_{H_0^1(0,1)} \leq ||\varphi_+ -
  I\varphi_+||_{H_0^1(0,1)} +||I \varphi_+ - P\varphi_+||_{H_0^1(0,1)} \\
  = ||\varphi_+ - I\varphi_+||_{H_0^1(0,1)} + ||P(I\varphi_+ -
  \varphi_+)||_{H_0^1(0,1)} \leq (1+D)||\varphi_+ - I\varphi_+||_{H_0^1(0,1)}.
\end{multline}
The same technique as in \eqref{eq:stableproj} can also be used for
the $L^2$ norm, now using the obvious bound $||P\varphi|| \leq
||\varphi||$, which implies that
\begin{equation}\label{eq:stableproj2}
||\varphi_+ - P \varphi_+|| \leq 2 ||\varphi_+ - I\varphi_+||.
\end{equation}
The equations \eqref{eq:beginningsemiconres}, \eqref{eq:interpbounds},
\eqref{eq:stableproj} and \eqref{eq:stableproj2} imply that also the
projection part of \eqref{eq:comparisonphi+} is correct.
\end{proof}
Theorems \ref{thm:spatialconvergence} and \ref{thm:comparisonphi+}
directly imply the following corollary.
\begin{cor}\label{cor:truespaceconv}
There exists a constant $D$, such that
\begin{equation*}
|u(\varphi_+,0) - \bar u(P\varphi_+,0)| + |u(\varphi_+,0) - \bar
 u(I\varphi_+,0)| \leq D \Delta x^2.
\end{equation*}
\end{cor}
\section{Discretization in time}\label{sec:timediscr}
In \cite{Sandberg-Szepessy} the method \emph{Symplectic Pontryagin} for
approximation of optimally controlled ODE:s is constructed and
analyzed. It is a Symplectic Euler discretization for a Hamiltonian
system, involving the state and
dual variables associated with the control problem, with a regularized
Hamiltonian. In the present situation, when the Hamiltonian is smooth,
the Symplectic Pontryagin method reduces to ordinary Symplectic Euler,
since no need for regularization exists. The theory in
\cite{Sandberg-Szepessy} can be used to show that the difference
between the value function for a system with only spatial
discretization and the value function for a system with discretization
in space \emph{and} time, is of the order $\Delta t$, where $\Delta t$
is the size of the temporal discretization. 
It is, however, desirable to achieve more than this. In order for the
estimate on the temporal discretization to be useful the constant in
front of $\Delta t$ in the error estimate needs to be really constant,
i.e.\ independent of $\Delta x$.

The theorems in \cite{Sandberg-Szepessy} do not directly provide the
desired result. This has to do with the fact that the second order
difference quotient operator $D^2$, defined in
\eqref{eq:diffquotient}, has norm proportional to $1/\Delta
x^2$. Furthermore, the proof in \cite{Sandberg-Szepessy} requires a
bound on the derivatives
$\partial\tilde\lambda^{n+1}/\partial\tilde\varphi^n$, where
$\tilde\varphi$ and $\tilde\lambda$ are obtained with the Symplectic
Pontryagin method. The problem of large norm of $D^2$ can be handled
using that it is a negative operator. But in addition to this we also
need to bound $\partial\tilde\lambda^{n+1}/\partial\tilde\varphi^n$
independently of $\Delta x$ in some proper sense.

The proof of convergence of the Symplectic Euler method given here is
based on another technique. It uses that the present problem admits
optimal controls which are regular by Theorem
\ref{thm:diffcontrol}. It also involves an assumption about the
derivative $\tilde\varphi_x$, and another similar assumption. Under
these assumptions it is shown in Theorem \ref{thm:timediscrconv} that
a minimum of a  forward Euler approximation of
control problem \eqref{eq:approxu}, \eqref{eq:approxphievol} has an
error $C\Delta t$ in the objective, where $C$ does not depend on
$\Delta x$. In Theorem \ref{thm:sympleuler} it is shown that the
solution to this minimization problem is equivalent to the solution of
a Symplectic Euler scheme, and hence the desired property for the
Symplectic Euler scheme is achieved. The main difference in the result
when the present method is used compared to a result using the theory
in \cite{Sandberg-Szepessy} is that the present result needs an
assumption on the derivative $\tilde\varphi_x$ whereas
\cite{Sandberg-Szepessy} needs control over
$\partial\tilde\lambda^{n+1}/\partial \tilde \varphi^n$. The
assumptions on  $\tilde\varphi_x$ seem easier to verify. The numerical
tests performed in Section \ref{sec:NumRes} support that it is true.

We now present the setting of the aforementioned discretized
optimization problem. Consider the time-discrete state
$\{\tilde\varphi^n\}_{n=0}^N$, which is a forward Euler approximation of
the state $\bar\varphi$ in \eqref{eq:approxphievol} and is given by
\begin{equation}\label{eq:timediscretephievol}
(\tilde\varphi^{n+1},v)=(\tilde\varphi^n,v)+\Delta t \big( -\delta
  (\tilde\varphi^n_x,v_x) +(-\delta^{-1}V'(\tilde\varphi^n) +
  \tilde\alpha^n,v)\big), \quad \text{for all } v\in V,
\end{equation} 
where $\{\tilde\alpha^n\}_{n=0}^{N-1}$ is a time-discrete control. The
discrete state $\tilde\varphi^n$ therefore corresponds to
$\bar\varphi(t_n)$, where $t_n=n \frac{T}{N} \equiv n \Delta t$. By
\eqref{eq:timediscretephievol} it is possible to define a discrete value
function for all times $t_m$:
\begin{equation}\label{eq:timediscretevalue}
\tilde u (\tilde\varphi_0,t_m)=\min_{\{\tilde\alpha^n\}_{n=m}^{N-1}}
\big( g(\tilde\varphi_N)+ \Delta t \sum_{n=m}^{N-1} h(\tilde \alpha^n)\big),
\end{equation}
where $\{\tilde\varphi^n\}$ solves \eqref{eq:timediscretephievol} and
$\tilde\varphi^m=\tilde\varphi_0$. For the proof of Theorem
\ref{thm:timediscrconv} we also introduce the 
discrete state $\{\mathring\varphi^n\}_{n=0}^{N}$. It
is also given by a forward Euler time stepping scheme, but its
evolution is determined by an optimal control $\bar\alpha$ to the
time-continuous problem \eqref{eq:approxphievol}:
\begin{equation}\label{eq:timediscretephievol2}
(\mathring\varphi^{n+1},v)=(\mathring\varphi^n,v)+\Delta t \big( -\delta
  (\mathring\varphi^n_x,v_x) +(-\delta^{-1}V'(\mathring\varphi^n) +
  \bar\alpha(t_n),v)\big), \ \text{for all } v\in V.
\end{equation}
We will consider starting positions $\bar\varphi_0$ in finite
element spaces $V$ satisfying
\begin{equation}\label{eq:startingpositions}
\delta\big((\bar\varphi_0)_x,v_x\big)+\delta^{-1}\big(V'(\bar\varphi_0\big),v)=0, \quad
\text{for all } v \in V.
\end{equation}
We are now ready for the theorem on time discretization convergence.
\begin{thm}\label{thm:timediscrconv}
Assume there exists a function $r:\Re^+ \rightarrow \Re^+$ such that
for all $\Delta t \leq r(\Delta x)$ there are solutions
$\{\tilde\varphi^n\}_{n=0}^N$ and $\{\mathring\varphi^n \}_{n=0}^N$ with 
$\tilde\varphi^0=\mathring\varphi^0=\bar\varphi_0$, where $\bar\varphi_0$
satisfies \eqref{eq:startingpositions}, and 
\begin{equation}\label{eq:gradincrease}
||\tilde\varphi^{n+1}_x-\tilde\varphi^n_x|| +
  ||\mathring\varphi^{n+1}_x-\mathring\varphi^n_x|| \leq C \Delta t 
\end{equation}
for all $0 \leq n < N$, where $C$ does not depend on $\Delta x$. Then
\begin{equation*}
|\tilde u(\bar\varphi_0,0)- \bar u(\bar\varphi_0,0) | \leq D \Delta t
\end{equation*}
for $\Delta t \leq r(\Delta x)$, where $D$ does not depend on $\Delta x$.
\end{thm}
\begin{remark}
By the numerical computations performed in Section \ref{sec:NumRes} it
seems plausible that \eqref{eq:gradincrease} holds.
\end{remark}
\begin{remark}
The proof would be valid without inclusion of the function $r$. 
However, since the forward Euler method is used it seems reasonable to
believe that \eqref{eq:gradincrease} would not be valid for all
$\Delta t$.
\end{remark}
\begin{proof}
As for Theorem \ref{thm:valueerror} the proof is divided into two
steps. We obtain in the first step a lower bound for $\tilde
u(\bar\varphi_0,0)- \bar u (\bar\varphi_0,0)$, and in the second step a
corresponding upper bound. The first step in this proof is similar to
the first step in the proof of Theorem \ref{thm:valueerror}, while the
corresponding second steps differ. We denote an optimal pair (control
and state) for $\bar u$ by $\bar\alpha$ and $\bar\varphi$, and an optimal
pair for $\tilde u$ by $\{\tilde\alpha^n\}$ and $\{\tilde\varphi^n \}$.

\emph{Step 1.} This part of the proof starts by an extension of
the initially time-discrete state $\{\tilde\varphi^n\}$ to a piecewise
linear time-continuous function $\tilde\varphi : [0,T] \rightarrow V$ as
follows:
\begin{equation*}
\tilde\varphi(t) \equiv \frac{t_{n+1}-t}{\Delta t}\tilde\varphi^n +
\frac{t-t_n}{\Delta t} \tilde\varphi^{n+1}, \quad \text{for } t_n \leq t
\leq t_{n+1}.
\end{equation*}
As in the proof of Theorem \ref{thm:valueerror} we have 
\begin{equation}\label{eq:udiff}
\tilde u(\bar\varphi_0,0)-\bar u (\bar\varphi_0,0) = \int_0^T \frac{d}{ds}
\bar u (\tilde\varphi(s),s)ds + \Delta t \sum_{i=0}^{N-1} h(\tilde \alpha^i).
\end{equation}
In order to be able to use that $\bar u$ solves a Hamilton-Jacobi
equation we note that the right hand side in \eqref{eq:udiff} may be
written
\begin{equation*}
\sum_{i=0}^{N-1} \int_{t_i}^{t_{i+1}} \big(\frac{d}{ds} \bar u
(\tilde\varphi(s),s)+h(\tilde\alpha^i) \big)ds,
\end{equation*}
and thus we may focus our attention on one time interval
$[t_n,t_{n+1}]$.
We also note that equation \eqref{eq:approxphievol} defines a flow
$\bar f:V\times V \rightarrow V$ which is defined by
\begin{equation}\label{eq:fbar}
(\bar f(\bar\varphi,\bar\alpha),v)= -\delta(\bar\varphi_x,v_x) +
  (\delta^{-1} V'(\bar\varphi)+\bar\alpha,v), \quad \text{for all } v \in V.
\end{equation} 
%Similarly as the further development in the proof of Theorem
%\ref{thm:valueerror} the time derivative in \eqref{eq:udiff} is for
%$t_n < s < t_{n+1}$
%\begin{align*}
%\frac{d}{ds}\bar u(\tilde\varphi(s),s) &= \min_{p \in D^+ \bar u
%  (\tilde\varphi(s),s)} \Big(p_t + \big( p_{\varphi},\bar
%  f(\tilde\varphi^n,\tilde\alpha^n)\big)\Big) \\
% & \equiv p^*_t(s) + \big( p^*_\varphi (s), \bar f(\tilde\varphi^n,\tilde\alpha^n)\big).
%\end{align*}
Let now $p(s)=\big(p_\varphi(s),p_t(s)\big)$ be any element in $D^+
\bar u\big(\tilde\varphi(s),s\big)$.
Similarly as in the proof of Theorem \ref{thm:valueerror} we have for
almost every  $s \in [t_n,t_{n+1}]$
\begin{equation}\label{eq:fdiff}
\begin{split}
&\frac{d}{ds}\bar u (\tilde\varphi(s),s) +
  h(\tilde\alpha^i)
\geq p_t(s) + \big(p_\varphi(s),\bar f(\tilde\varphi^n,\tilde\alpha^n)\big) 
\\
& =p_t(s) +
\underbrace{\big( p_\varphi (s), \bar f(\tilde\varphi(s),\tilde\alpha^n)\big) +
h(\tilde\alpha^n)}_{\geq H(p_{\varphi}(s),\tilde\varphi(s))} + \big(p_\varphi(s),\bar
f(\tilde\varphi^n,\tilde\alpha^n)-\bar f(\tilde\varphi(s),\tilde\alpha^n)
\big) \\
&\geq \big(p_\varphi(s),\bar
f(\tilde\varphi^n,\tilde\alpha^n)-\bar f(\tilde\varphi(s),\tilde\alpha^n)
\big),
\end{split}
\end{equation}
since $p_t(s)+H(p_\varphi(s),\tilde\varphi(s)) \geq 0$ as $\bar u $ is a
Hamilton-Jacobi viscosity solution. By assumption
\eqref{eq:gradincrease} it follows that $\tilde\varphi_x(s)$ is bounded
for $0 \leq s \leq T$ independently of $\Delta x$. We are therefore
free to use $\tilde\varphi(s)$  as
$\bar\varphi_0$ in Theorem \ref{thm:diffcontrol} so that the
$\bar\lambda(s)$ (corresponding to $\bar u (\tilde\varphi(s),s)$) is
bounded in $H_0^1$ independently of $\Delta x$. For such a
$\bar\lambda(s)$ we have
\begin{align*}
 & \big | \big(\bar\lambda(s),\bar f(\tilde\varphi^n,\tilde\alpha^n)-\bar
 f(\tilde\varphi(s),\tilde\alpha^n)\big) \big| \\
& \leq \delta \big| \big(\bar\lambda_x(s),\tilde\varphi^n_x-\tilde\varphi_x(s)\big) \big| +
 \delta^{-1} \big|\big(V'(\tilde\varphi^n)-V'(\tilde\varphi(s)),\bar\lambda(s)\big)\big|
 \leq C \Delta t,
\end{align*} 
with $C$ independent of $\Delta x$ by \eqref{eq:gradincrease}. 
It is now used that $\bar\lambda(s)$ is the spatial part of an element
in $D^+ \bar u(\tilde\varphi(s),s)$.
It thereby   holds that the right hand side in \eqref{eq:fdiff}
is less than $C\Delta t$ in magnitude.

\emph{Step 2.} We start by noting that
\begin{equation}\label{eq:bound2}
\bar u(\bar\varphi_0,0)-\tilde u (\bar\varphi_0,0) \geq
g(\bar\varphi(T))+\int_0^T h(\bar\alpha)dt - \big(
g(\mathring\varphi^N)+\Delta t \sum_{i=0}^{N-1} h(\bar\alpha(t_i)) \big).
\end{equation}
The difference between the running costs in \eqref{eq:bound2} is 
\begin{equation*}
\sum_{i=0}^{N-1} \int_{t_i}^{t_{i+1}} \big( h(\bar\alpha(t))-h(\bar\alpha(t_n))\big)dt.
\end{equation*}
Using that $h(\alpha)=||\alpha||^2/2$ we have that 
\begin{multline*}
|h(\bar\alpha(t))-h(\bar\alpha(t_n))| = \half
 |(\bar\alpha(t)+\bar\alpha(t_n),\bar\alpha(t)-\bar\alpha(t_n))| \\
\leq
 \half
 ||\bar\alpha(t)+\bar\alpha(t_n)||\cdot||\bar\alpha(t)-\bar\alpha(t_n)|| \leq C\Delta t,
\end{multline*}
where we have used the result in Theorem \ref{thm:diffcontrol} on the
boundedness of the control and its derivative (remember that
$\bar\alpha= - \bar\lambda$). 
It remains to show that the difference between the terminal costs in
\eqref{eq:bound2} behaves similarly. As in \emph{Step 1} we now extend
the discrete state $\{\mathring \varphi^n\}$ to a continuous function:
\begin{equation*}
\mathring\varphi(t) \equiv \frac{t_{n+1}-t}{\Delta t}\mathring\varphi^n +
\frac{t-t_n}{\Delta t} \mathring\varphi^{n+1}, \quad \text{for } t_n \leq t
\leq t_{n+1}.
\end{equation*}
For $t_n < t < t_{n+1}$ the evolution equations for $\bar\varphi$ and
$\mathring\varphi$ look as follows:
\begin{align*}
(\bar\varphi_t,v)&=-\delta(\bar\varphi_x,v_x)+(-\delta^{-1}V'(\bar\varphi)
  +\bar\alpha,v), \\
(\mathring\varphi_t,v)&=-\delta(\mathring\varphi_x^n,v_x)+(-\delta^{-1}V'(\mathring\varphi^n)
  +\bar\alpha(t_n),v), 
\end{align*}
for all $v \in V$. Subtract these two equations and let
$v=\bar\varphi-\mathring\varphi$ to get:
\begin{align*}
&\half\frac{d}{dt} ||\bar\varphi - \mathring\varphi||^2 \\
&= -\delta
(\bar\varphi_x - \mathring\varphi_x^n,\bar\varphi_x -\mathring\varphi_x) +\delta^{-1}(V'(\mathring\varphi^n)-V'(\bar\varphi),\bar\varphi -\mathring\varphi)\\
&\quad\quad\quad + (\bar\alpha - \bar\alpha(t_n),\bar\varphi -\mathring\varphi) \\
&=-\delta ||\bar\varphi_x-\mathring\varphi_x||^2 +
\delta(\mathring\varphi_x^n-\mathring\varphi_x,\bar\varphi_x-\mathring\varphi_x)
\\
&\quad\quad\quad +\delta^{-1}(V'(\mathring\varphi^n)-V'(\bar\varphi),\bar\varphi -\mathring\varphi)
+ (\bar\alpha - \bar\alpha(t_n),\bar\varphi -\mathring\varphi) \\
& \leq -\delta ||\bar\varphi_x-\mathring\varphi_x||^2 +
\delta\frac{||\mathring\varphi^n_x - \mathring\varphi_x||^2}{2} +
\delta\frac{||\bar\varphi_x - \mathring\varphi_x||^2}{2} \\
&\quad + \delta^{-1} |V''|
\cdot ||\mathring\varphi^n - \mathring\varphi + \mathring\varphi-\bar\varphi||\cdot
||\bar\varphi- \mathring\varphi|| + \frac{||\bar\alpha-
  \bar\alpha(t_n)||^2}{2} + \frac{||\bar\varphi -\mathring\varphi||^2}{2} \\
& \leq \delta\frac{||\mathring\varphi^n_x - \mathring\varphi_x||^2}{2} +\frac{\delta^{-1}|V''|}{2}||\mathring\varphi^n - \mathring\varphi||^2 \\
&\quad\quad\quad +\frac{\delta^{-1}|V''| + 1}{2} ||\bar\varphi -\mathring\varphi||^2 + \frac{||\bar\alpha-
  \bar\alpha(t_n)||^2}{2}.
\end{align*}
According to a Poincar\'e inequality (see e.g.\ Theorem 5.3.5 in
\cite{Brenner-Scott}), using the Dirichlet conditions, we have that 
$||\mathring\varphi^n - \mathring\varphi|| \leq C||\mathring\varphi^n_x -
\mathring\varphi_x||$. If now Gr\"onwall's Lemma is  used together with
the fact that
$||\bar\alpha-\bar\alpha^n||+||\mathring\varphi^n_x - \mathring\varphi_x||
\leq C \Delta t$, we have that $||\bar\varphi(T)-\mathring\varphi(T)||\leq
\Delta t$. Since $g(\bar\varphi(T))$ is bounded independently of $\Delta
x$ we have, similarly as in the proof of Theorem
\ref{thm:boundedcontrol},  that
$|g(\bar\varphi(T))-g(\mathring\varphi(T))|\leq C \Delta t$.
\end{proof}
As convergence of the forward Euler method has now been proved, the
\emph{Symplectic Euler method}, which
can be used to find the forward Euler solution, is now presented. It
is given by the system
\begin{equation}\label{eq:sympleuler}
\begin{split}
(\tilde\varphi^{n+1},v) &= (\tilde\varphi^n,v) + \Delta t H_\lambda
  (\tilde\lambda^{n+1},\tilde\varphi^n;v) \\
&= (\tilde\varphi^n,v)+\Delta t
  \big( -\delta (\tilde\varphi^n_x,v_x)- \delta^{-1}(V'(\tilde\varphi^n),v) 
  - (\tilde\lambda^{n+1},v)\big), \\
\tilde\varphi^0 & =\bar\varphi_0, \\
(\tilde\lambda^n,v) &= (\tilde\lambda^{n+1},v) + \Delta t
  H_\varphi(\tilde\lambda^{n+1},\tilde\varphi^n;v) \\
& = (\tilde\lambda^{n+1},v)
  + \Delta t
  \big(-\delta(\tilde\lambda^{n+1}_x,v_x)-\delta^{-1}(\tilde\lambda^{n+1}V''(\tilde\varphi^n),v)\big),
  \\
\tilde\lambda^N & = g'(\tilde\varphi^N)=2K(\tilde\varphi^N-P \varphi_-),
\end{split}
\end{equation} 
where $g'$ is a G\^ateaux derivative and $H_\lambda(\cdot;v)$,
$H_\varphi(\cdot;v)$ are G\^ateaux derivatives in the direction $v$. For
every minimizer $\{\tilde\alpha^n\}$ in  \eqref{eq:timediscretevalue}
there exists a solution to \eqref{eq:sympleuler} with
$\tilde\lambda^{n+1} = -\tilde\alpha^n$ for all $n$. In order to prove
this we first state a lemma.
\begin{lemma}\label{lem:discretesemiconcavity}
The value function $\tilde u(\cdot,t_n)$ is semiconcave for every $n$.
\end{lemma}
\begin{proof}
Consider the starting positions $\tilde\varphi^0_1$, $\tilde\varphi^0_2$ and
$\frac{\tilde\varphi^0_1+\tilde\varphi^0_2}{2}$ at time $0$. The
time-discrete cost functional $\tilde v$ is introduced:
\begin{equation*}
\tilde v_{\tilde\varphi_0,t_m}(\{\tilde\alpha^n\}) = \big( g(\tilde\varphi_N)+ \Delta t \sum_{n=m}^{N-1} h(\tilde \alpha^n)\big),
\end{equation*}
where $\{\tilde\varphi^n\}$ solves \eqref{eq:timediscretephievol} and
$\tilde\varphi_m=\tilde\varphi_0$.
Let $\{\tilde\alpha^n\}$ be an optimal control for the starting
position
$(\frac{\tilde\varphi^0_1+\tilde\varphi^0_2}{2},0)$. We can thus write
\begin{multline*}
\tilde u(\tilde\varphi^0_1,0)+\tilde u(\tilde\varphi^0_2,0) - 2\tilde
u(\frac{\tilde\varphi^0_1+\tilde\varphi^0_2}{2},0) \\ 
\leq \tilde
v_{\tilde\varphi^0_1,0}(\{\tilde\alpha^n\}) + \tilde
v_{\tilde\varphi^0_2,0}(\{\tilde\alpha^n\}) -2 \tilde
v_{\frac{\tilde\varphi^0_1+\tilde\varphi^0_2}{2},0}(\{\tilde\alpha^n\}).
\end{multline*}
The states starting in $\tilde\varphi^0_1$, $\tilde\varphi^0_2$ and
$\frac{\tilde\varphi^0_1+\tilde\varphi^0_2}{2}$, all using the control
$\{\tilde\alpha^n\}$, are called
\begin{equation*}
\tilde\varphi^n_1 \equiv \sum_{i=1}^{M-1}\xi^n_{1,i} v^i,\  \tilde\varphi^n_2
\equiv \sum_{i=1}^{M-1}\xi^n_{2,i} v^i, \text{ and } \tilde\varphi^n_3 \equiv \sum_{i=1}^{M-1}\xi^n_{3,i} v^i.
\end{equation*}
Introducing the notation
\begin{equation*}
\xi^n_m= \begin{pmatrix} \xi^n_{m,1} \\ \vdots \\ \xi^{n}_{m,M-1}
\end{pmatrix}, \quad
p^n_m = \begin{pmatrix} \big(V'(\tilde\varphi_m^n),v^1\big) \\ \vdots \\
  \big(V'(\tilde\varphi_m^n),v^{M-1}\big)\end{pmatrix}, \quad
a^n = \begin{pmatrix} a^n_1 \\ \vdots \\ a^{n}_{M-1} \end{pmatrix},
\end{equation*}
where $m$ can be 1, 2, or 3 and
\begin{equation*}
\tilde\alpha^n \equiv \sum_{i=1}^{M-1}a^n_i v^i,
\end{equation*}
we can, using the mass matrix $B$ in \eqref{eq:massmatrix} and the second difference
operator $D^2$ in \eqref{eq:diffquotient}, write the equation for $\tilde \varphi^n_m$,
$m=1,2,3$, as follows:
\begin{equation}\label{eq:xistates}
B \xi^{n+1}_m = B \xi^n_m + \Delta t \big(\delta D^2 \xi^n_m -
\frac{\delta^{-1}}{\Delta x}p^n_m + B a^n\big).
\end{equation} 
Introducing the state $z^n = \xi^n_1 + \xi^n_2 - 2\xi^n_3$ and
using \eqref{eq:xistates} gives
\begin{equation}\label{eq:discrzevol}
B z^{n+1}=B z^n +\Delta t\big(\delta D^2 z^n -
\frac{\delta^{-1}}{\Delta x}(p^n_1+p^n_2-2p^n_3)\big).
\end{equation}
Every element in the vector 
\begin{equation*}
p^n_1+p^n_2-2p^n_3 = \begin{pmatrix} \big( V'(\tilde\varphi^n_1) +
  V'(\tilde\varphi^n_2) -2 V'(\tilde\varphi^n_3),v^1 \big) \\ \vdots \\ \big( V'(\tilde\varphi^n_1) +
  V'(\tilde\varphi^n_2) -2 V'(\tilde\varphi^n_3),v^{M-1} \big)
\end{pmatrix}
\end{equation*}
can be bounded in magnitude by
\begin{multline*}
||V'(\tilde\varphi^n_1)+V'(\tilde\varphi^n_2)-2
  V'(\tilde\varphi^n_3)|| \cdot ||v^i||\\
 = \sqrt{\frac{2}{3}\Delta x} ||V'(\tilde\varphi^n_1)+V'(\tilde\varphi^n_2)-2
  V'(\tilde\varphi^n_3)||\\
=\sqrt{\frac{2}{3}\Delta x} ||V'(\tilde\varphi^n_1)+V'(\tilde\varphi^n_2)-2
  V'(\frac{\tilde\varphi^n_1+\tilde\varphi^n_2}{2})|| \\+ 2\sqrt{\frac{2}{3}\Delta x}||
  V'(\frac{\tilde\varphi^n_1+\tilde\varphi^n_2}{2}) -
  V'(\tilde\varphi^n_3)|| =: I + II,
\end{multline*}
using the triangle inequality as in \eqref{eq:Vsplit}.  We first treat
term $I$ above:
\begin{multline*}
||V'(\tilde\varphi^n_1)+V'(\tilde\varphi^n_2)-2
  V'(\frac{\tilde\varphi^n_1+\tilde\varphi^n_2}{2})||^2 \leq
  \frac{|V'''|^2}{2^2} \int_0^1 |\tilde\varphi^n_1 -
  \tilde\varphi^n_2|^4 dx \\
\leq \frac{|V'''|^2}{2^2} \big( \max
  |\tilde\varphi^n_1 - \tilde\varphi^n_2| \big)^4 =
  \frac{|V'''|^2}{2^2} \big( \max |\xi^n_1 - \xi^n_2| \big)^4 \\
\leq
\frac{|V'''|^2}{2^2}  \big(\sum_{i=1}^{M-1} |\xi^n_{1,i} - \xi^n_{2,i}|^2 \big)^2 = \frac{|V'''|^2}{2^2}||
  \xi^n_1 - \xi^n_2 ||^4_2,
\end{multline*}
where $||\cdot||_2$ denotes the Euclidean vector norm.
Part $II$ may be bounded as follows:
\begin{equation*}
2||
  V'(\frac{\tilde\varphi^n_1+\tilde\varphi^n_2}{2}) -
  V'(\tilde\varphi^n_3)|| \leq A||\tilde\varphi^n_1 + \tilde\varphi^n_2
  - 2 \tilde\varphi^n_3|| \leq B ||z^n||_2.
\end{equation*}
These facts in \eqref{eq:discrzevol} give that 
\begin{equation}\label{eq:znorm}
||z^{n+1}||_2 \leq C||z^n||_2 + D ||\xi^n_1-\xi^n_2||^2_2.
\end{equation}
By subtracting the equations for $m=1$
and $m=2$ in \eqref{eq:xistates} we see that 
\begin{equation*}
||\xi^{n+1}_1-\xi^{n+1}_2||_2 \leq E ||\xi^{n}_1-\xi^{n}_2||_2 \leq
  \ldots \leq E^{n+1}||\xi^0_1-\xi^0_2||_2,
\end{equation*}
so that in \eqref{eq:znorm} we could really write
$D||\xi^0_1-\xi^0_2||^2_2$ instead of
$D||\xi^n_1-\xi^n_2||^2_2$. Thereby, since $z^0=0$, it holds that
$||z^N||_2 \leq F||\xi^0_1 - \xi^0_2||_2^2$. 
Note that the constants $A$ -- $F$ are allowed to depend on
$\Delta x$.
Similarly as in the proof
of Theorem \ref{thm:semiconcave}, semiconcavity  of $\tilde u$ is a
consequence of this.
\end{proof}
We are now ready for the promised theorem about the Symplectic Euler
method.
\begin{thm}\label{thm:sympleuler}
For every minimizer $\{\tilde\alpha^n\}$ in \eqref{eq:timediscretevalue} there exists a
solution to \eqref{eq:sympleuler} with $\tilde\lambda^{n+1}=-\tilde\alpha^n$ for $0
\leq n \leq N-1$.
\end{thm}
\begin{proof}
The proof is divided into three steps. In the first step it is shown
that the value function $\tilde u$ is differentiable along the optimal
path $\tilde\varphi^n$. In the second step it is proved that the dual
variable $\tilde\lambda^n$ equals the G\^ateaux derivative of $\tilde u$, and in
the last step it is shown that $\tilde\lambda^{n+1}=\tilde\alpha^n$.

\emph{Step 1.} In order to show that the discrete value function
$\tilde u$ is differentiable at $(\tilde\varphi^n,t_n)$ for $0 < n \leq
N$ the function
\begin{equation}\label{eq:r}
r(\tilde\alpha) \equiv \tilde u(\tilde\varphi,t_{n+1}) + \Delta t h(\tilde\alpha)
\end{equation}
is introduced, where $\tilde\varphi=\tilde\varphi^n + \Delta t \bar
f(\tilde\varphi^n,\tilde\alpha)$ and $\bar f$ is given by \eqref{eq:fbar}. 
Assume that $\tilde u$ is not differentiable at
$(\tilde\varphi^{n+1},t_{n+1})$. Because $\tilde u$ is semiconcave it then
follows that the superdifferential $D^+ \tilde
u(\tilde\varphi^{n+1},t_{n+1})$ (which we let designate the
superdifferentials in the G\^ateaux sense) contains more than one
point.  For all $\tilde \alpha$ in a neighborhood of $\tilde \alpha^n$
it holds that
\begin{equation}\label{eq:rdiff}
\begin{split}
&r(\tilde\alpha)-r(\tilde\alpha^n) \\
 =& \tilde u \big(\tilde\varphi^n + \Delta
t \bar f(\tilde\varphi^n,\tilde\alpha)\big) - \tilde u \big(\tilde\varphi^n + \Delta
t \bar f(\tilde\varphi^n,\tilde\alpha^n)\big) +\Delta t
\big(h(\tilde\alpha)-h(\tilde\alpha^n)\big) \\
 \leq  &\Delta t \big(p, \bar f(\tilde\varphi^n,\tilde\alpha)- \bar
f(\tilde\varphi^n,\tilde\alpha^n)\big) + \Delta t
\big(h'(\tilde\alpha^n),\tilde\alpha - \tilde\alpha^n\big) + K||\tilde\alpha -
\tilde\alpha^n||^2 \\
 = &\Delta t (\tilde p, \tilde\alpha- \tilde\alpha^n) + \Delta t
\big(h'(\tilde\alpha^n),\tilde\alpha - \tilde\alpha^n\big) + K||\tilde\alpha -
\tilde\alpha^n||^2,
\end{split} 
\end{equation}
where $p$ is an element in $D^+ \tilde
u(\tilde\varphi^{n+1},t_{n+1})$ and $\tilde p$ is given by a linear
bijection of $s$, since $\bar f$ is linear in the $\alpha$
variable. Since there are more than one element $p \in D^+ \tilde
u(\tilde\varphi^{n+1},t_{n+1})$, there are also more than one possible $\tilde p$
in equation \eqref{eq:rdiff}. 
It is therefore possible to choose the element $\tilde p$ such that
the linear term in \eqref{eq:rdiff} is non-vanishing.
It follows that there exists $\tilde
\alpha$ such that $r(\tilde\alpha) < r(\tilde\alpha^n)$, which is the
sought contradiction. By this reasoning we see that $\tilde u$ is
differentiable at $(\tilde\varphi^n,t_n)$ for $0<n \leq N$. 

\emph{Step 2.} It follows directly that
$\tilde\lambda^N=g'(\tilde\varphi^N)$, i.e.\ the G\^ateaux derivative of
$\tilde u(\cdot,t_N)$. Assume that $\tilde\lambda^{n+1} = \tilde
u_\varphi(\tilde\varphi^{n+1},t_{n+1})$. It will follow from this that
$\tilde\lambda^n =\tilde u_\varphi(\tilde\varphi^n,t_n)$. Since it is known
that $\tilde u$ is differentiable at both $(\tilde\varphi^n,t_n)$ and
$(\tilde\varphi^{n+1},t_{n+1})$ the G\^ateaux derivative of $\tilde u$ at
$(\tilde\varphi^{n+1},t_{n+1})$ equals the G\^ateaux derivative at
$\tilde \varphi^{n}$ of the function
\begin{equation*}
s(\tilde\varphi) \equiv \tilde  u (\tilde\varphi + \Delta t \bar
f(\tilde\varphi,\tilde\alpha^n),t_n) + \Delta t h(\tilde\alpha^n),
\end{equation*}
where $\tilde\alpha^n$ is fixed. The G\^ateaux derivative of $s$ at $\tilde\varphi^n$ is
given by 
\begin{equation*}
s'(\tilde\varphi^n) = \tilde u_\varphi(\tilde\varphi^{n+1},t_{n+1}) \circ (I + \Delta t
\bar f'(\tilde\varphi^n)),
\end{equation*}
where $\tilde u_\varphi(\tilde\varphi^{n+1},t_{n+1}) =
\tilde\lambda^{n+1}$ is a function from $V$ to $\Re$ and $\bar
f'(\tilde\varphi^n)$ is a function from $V$ to $V$. This equation
coincides with the $\tilde\lambda$ equation in \eqref{eq:sympleuler},
which gives that $\tilde\lambda^n=\tilde u_\varphi(\tilde\varphi^n,t_n)$.
By induction in $n$   it follows that $\tilde\lambda^n=\tilde
u_\varphi(\tilde\varphi^n,t_n)$ for $0 < n \leq N$.

\emph{Step 3.} Knowing that $\tilde u$ is differentiable at $(\tilde
\varphi^n,t_n)$ for $0 < n \leq N$ the function \eqref{eq:r} can be
differentiated. Since $\tilde\alpha^n$ is a minimizer of $r$ the
derivative at this argument must be zero:
\begin{equation*}
r'(\tilde\alpha^n)=\tilde\lambda^{n+1} \Delta t \circ I +\Delta t \tilde\alpha^n=0,
\end{equation*}
where it is used that $\tilde
u_\varphi(\tilde\varphi^{n+1},t_{n+1})=\tilde\lambda^{n+1}$, $\bar f_\alpha
= I$ and $h'(\tilde\alpha^n)=\tilde\alpha^n$. It follows that
$\tilde\lambda^{n+1} = -\tilde\alpha^n$ for $0 \leq n \leq N-1$.
\end{proof}
\section{Numerical Results}\label{sec:NumRes}
We here present some numerical results for the 
Symplectic Euler scheme for a finite difference
discretization of \eqref{eq:flow}, \eqref{eq:valuefunction}.
The numerics is performed in this setting, partly because it is
slightly simpler than using finite elements, partly because a finite
difference discretization is used in \cite{Weinan-Ren-Vanden-Eijnden}.
The system we will consider is therefore 
\begin{equation}\label{eq:systxieta}
\begin{split}
\xi^{n+1} &= \xi^n + \Delta t \big(\delta D^2 \xi^n - \delta^{-1}
V'(\xi^n) -\eta^{n+1}\big), \\
\eta^n &= \eta^{n+1} + \Delta t \big(\delta D^2 \eta^{n+1} -
\delta^{-1}\eta^{n+1}V'(\xi^n)\big), \\
\eta^N &= 2K(\xi^N-\xi^-),
\end{split}
\end{equation}
where $\xi^-$ is a finite difference approximation to $\varphi_-$, $D^2$ is defined in \eqref{eq:diffquotient} and $\xi^n$ and
$\eta^n$ correspond to the nodal values of $\bar\varphi^n$ and
$\bar\lambda^n$, respectively. The approximate value used together
with this scheme is
\begin{equation}\label{eq:FDvalue}
K \Delta x ||\xi^N-\xi^-||_2^2 + \Delta t\Delta x \sum_{n=1}^{n=N}
||\eta^n||_2^2 /2,
\end{equation}
where $||\cdot ||_2$ denotes the ordinary Euclidean vector norm.
As noted in \cite{Weinan-Ren-Vanden-Eijnden} there are several local
minima to \eqref{eq:FDvalue}, corresponding to different ``strategies'' to
overcome the potential barrier $V$. The switching between the two
stable points proceeds by ``nucleation'', which involves a large
control $\alpha$, followed by propagation of domain walls. In Figure
\ref{fig:phi2walls1wallBETAmilliardR6percentM200N2006times}  the transition is shown for the cases propagation of one
and two domain walls. 
\begin{figure}
\centering
\includegraphics[width=0.8\textwidth,height=5cm]{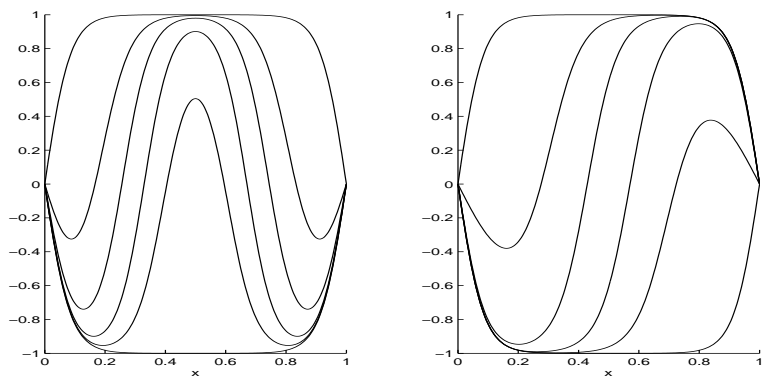}
\caption{Snapshots of transitions between the two stable
  configurations where $\varphi$ is shown at times 0, 0.2, 0.4, 0.6, 0.8 and 1 ($=T$).   To the right propagation of one wall and to the
  left propagation of two walls. In these examples $K=10^9$ and
  $\delta=0.06$ was used.}
\label{fig:phi2walls1wallBETAmilliardR6percentM200N2006times}
\end{figure}
The $\lambda$ variable, which equals the negative control, is shown in
Figure \ref{fig:lambda2wallsBETAmilliardR6percentM200N200} for the case of propagation of two walls.
\begin{figure}
\centering
\includegraphics[width=\textwidth]{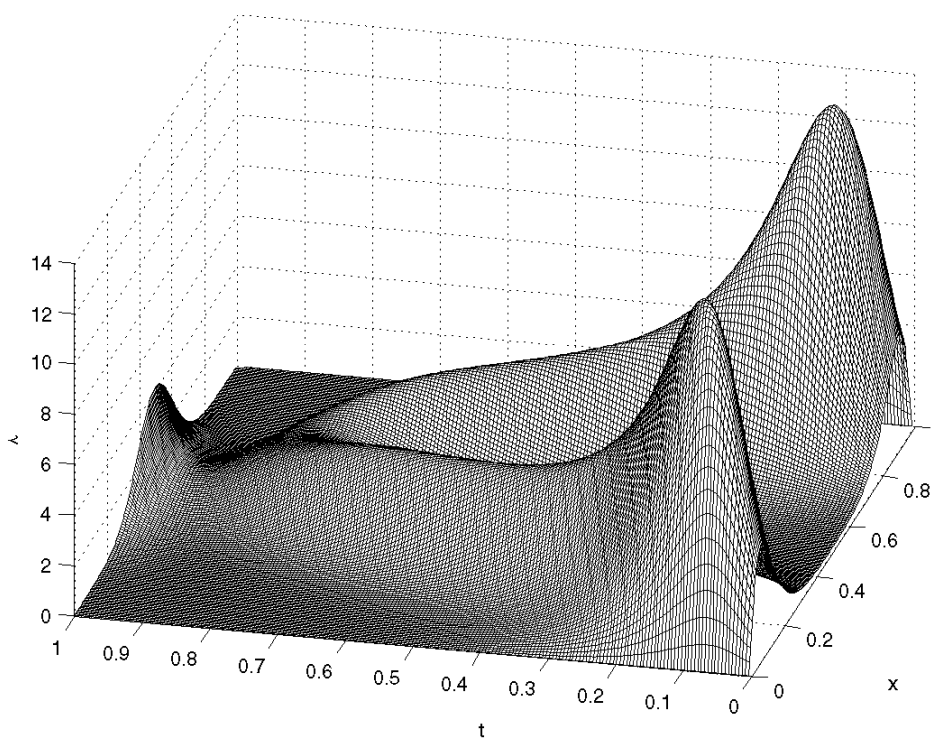}
\caption{The dual variable $\lambda$ for the case of two propagating
  walls corresponding to the left part of Figure \ref{fig:phi2walls1wallBETAmilliardR6percentM200N2006times}.}
\label{fig:lambda2wallsBETAmilliardR6percentM200N200}
\end{figure}

Apart from the Symplectic Forward Euler method previously mentioned,
the Symplectic Backward Euler method can also be used. This method is
given by
\begin{align*}
\xi^{n+1} &= \xi^n + \Delta t (\delta D^2 \xi^{n+1} - \delta^{-1}
V'(\xi^{n+1}) -\eta^{n}), \\
\eta^n &= \eta^{n+1} + \Delta t (\delta D^2 \eta^{n} -
\delta^{-1}\eta^{n}V'(\xi^{n+1})), \\
\eta^N &= 2K(\xi^N-\xi^-).
\end{align*} 
The approximate value for the Symplectic Backward Euler method is
given by (see Chapter 4.4 in \cite{Sandberg-Szepessy})
\begin{equation}\label{eq:FDBEvalue}
K \Delta x ||\xi^N-\xi^-||_2^2 + \Delta t\Delta x \sum_{n=0}^{n=N-1}
||\eta^n||_2^2 /2.
\end{equation}
An advantage with the Backward Euler method is that it enables using a
small $\Delta x$ even when $\Delta t$ is not small. This feature is
however not as profound for the present case of control of a
parabolic equation as for the uncontrolled case, as the control
compensates for the instability, which makes it possible to use smaller
$\Delta x$. Another good thing about the Backward Euler method is
that it seems to underestimate the optimal value while it seems to be
overestimated by the Forward Euler method.  Figure
\ref{fig:valueconvFEBEdxequalsdt} shows the dependence on $\Delta
x=\Delta t$ of the values \eqref{eq:FDvalue} and \eqref{eq:FDBEvalue}.
\begin{figure}
\centering
\includegraphics[width=\textwidth,height=5cm]{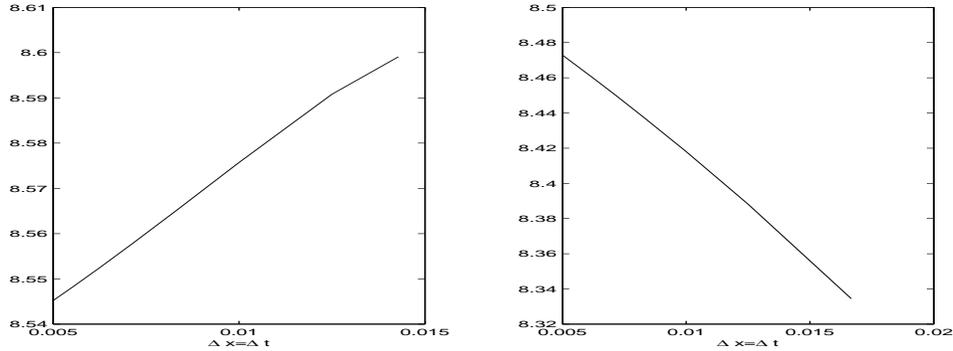}
\caption{Convergence of the optimal values \eqref{eq:FDvalue} and
  \eqref{eq:FDBEvalue} for the case of equal spacing in space and
  time, i.e.\ $\Delta x=\Delta t$. The left figure shows the values
  obtained by the Forward Euler method and the right shows the values
  of the Backward
  Euler method.}
\label{fig:valueconvFEBEdxequalsdt}
\end{figure}
By extrapolating these fairly straight curves to $\Delta x=\Delta t=0$
an approximate value of the optimal control problem is obtained. The
extrapolated value from the Forward Euler curve is 8.517, and the
approximate value from the Backward Euler curve is 8.526. 

We now indicate the dependence of the spatial discretization error on
$\Delta x$. This is done by changing the spatial discretization
$\Delta x$ while keeping the time discretization $\Delta t$
constant. We let the value
obtained for the smallest spatial discretization $\Delta x$ be the
reference value which takes the role of an ``exact'' solution. A convergence plot can be found in Figure
\ref{fig:dxconvergence2wallsBETAmilliardR3percentN200}. The slope of
the upper part of this curve corresponds to a convergence rate of
approximately $(\Delta x)^{2.37}$.
\begin{figure}
\centering
\includegraphics[width=0.5\textwidth]{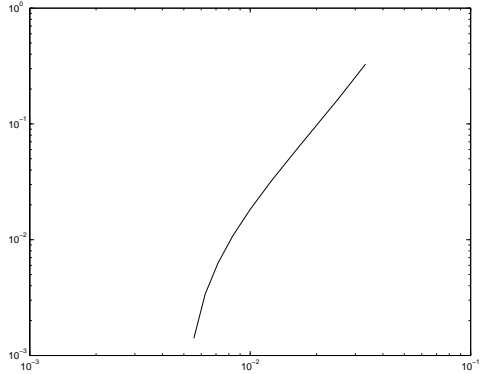}
\caption{Convergence of the optimal value (y-axis) with respect to
  $\Delta x$ (x-axis). The case with two propagating domain walls and
  $\Delta t=1/200$, $\delta=0.03$ and $K=10^9$.}
\label{fig:dxconvergence2wallsBETAmilliardR3percentN200}
\end{figure}

For the time discretization error we want to show that it is less than
a linear function of $\Delta t$ with a constant which does not depend
on $\Delta x$.  Time discretization convergence is therefore
considered for two spatial discretizations, one having $\Delta x=1/30$
and the other $\Delta x =1/100$. Since the Forward and Backward
Euler methods in the limit $\Delta t \rightarrow 0$ shall have the
same value we may extrapolate the values from diagrams similar to the
ones in Figure \ref{fig:valueconvFEBEdxequalsdt}, but with the
exception that $\Delta x$ is held fixed. The following  values are obtained from
these extrapolations in the case of two propagating domain walls,
using $\delta=0.03$ and $K=10^9$: 
\begin{itemize}
\item Forward Euler, $\Delta x=1/30$:\quad 8.841 
\item Forward Euler, $\Delta x=1/100$:\quad 8.547
\item Backward Euler, $\Delta x=1/30$:\quad 8.849
\item Backward Euler, $\Delta x=1/100$:\quad 8.555   
\end{itemize}
The mean of the above values for Forward and Backward Euler can be
taken as an ``exact'' reference value when convergence is
studied. Hence for
$\Delta x=1/30$ the reference value is taken to be 8.845 and for
$\Delta x=1/100$ it is taken to be 8.551. The two convergence plots
can be found in Figure
\ref{fig:dtconvergence2wallsBETAmilliardR3percentM30M100}. Note that
the inclination in the left curve, the values using $\Delta x=1/30$, is
larger than the inclination in the right curve ($\Delta
x=1/100$). This is in harmony with Theorem \ref{thm:timediscrconv} since
it is allowed that (and good if) we have faster convergence for smaller $\Delta x$.  
\begin{figure}
\centering
\includegraphics[width=\textwidth,height=5cm]{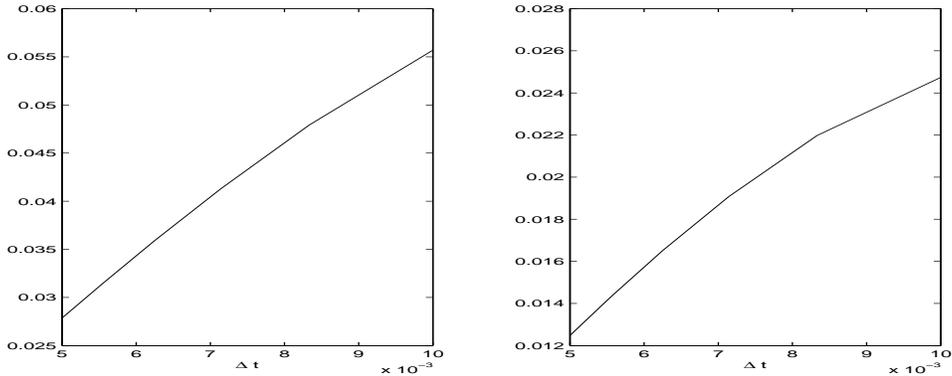}
\caption{Time discretization convergence for two different spatial
  discretizations, $\Delta x=1/30$ (left), and $\Delta x=1/100$
  (right).
The case with two propagating domain walls and
  $\delta=0.03$ and $K=10^9$.}
\label{fig:dtconvergence2wallsBETAmilliardR3percentM30M100}
\end{figure}

The system \eqref{eq:systxieta} can be (and has been) solved in two
steps. The first step gives a starting position for the second step,
and may be performed on a coarse grid, i.e.\ using large $\Delta x$
and $\Delta t$. The method is to choose an initial guess $\xi_0$ (a
vector containing all time steps) and with it compute the dual,
$\eta_0$, using \eqref{eq:systxieta}. This computed $\eta_0$ is used
in \eqref{eq:systxieta} to compute $\xi_{upd}$, an updated
$\xi$. Using a damping $\nu$, a new state $\xi_1=\nu \xi_0 +
(1-\nu)\xi_{upd}$ is computed which is used to obtain the dual
$\eta_1$, which in turn is used to compute a new $\xi_{upd}$, and so
on. When the difference $\xi_{upd}-\xi_n$ is sufficiently small the
iterations are terminated, and a starting point $(\xi,\eta)$ is
obtained for the second step.

Step two consists of Newton iterations of \eqref{eq:systxieta}. Since
the sparse Jacobian can be computed explicitly this second step
converges at a quadratic rate, making it computationally cheap to
reach  an accurate solution. In the examples presented in this chapter the
Newton iterations continued until the difference between two
consecutive $\xi$:s and $\eta$:s was less than $10^{-13}$ in each
space-time component. After convergence has been reached for some
discretization, a space-time interpolation of $\xi$ and $\eta$ can be
used as a starting position for a Newton iteration on a new grid. It
is also possible to gradually change the parameters $\delta$ and $K$
in the Newton iterations in order to be able to treat a favorite
case. When the starting point is sufficiently good the Newton
method terminates after 5-7 iteration steps, making it fast. As
comparison, when in \cite{Weinan-Ren-Vanden-Eijnden} a limited memory
BFGS method is used, about 550 iterations is needed to decrease the
$L_2$-norm of the objective gradient to $10^{-10}$, even when a clever
approximation of the initial Hessian was used.
\section{Acknowledgments}
I would like to thank Anders Szepessy and Erik von Schwerin for
proofreading this article and suggesting improvements. 
\section{Appendix}\label{sec:appendix}
In order to show existence and uniqueness of solutions to
(\ref{eq:flow}) we introduce the notion of \emph{weak} solutions (see
\cite{Evans}). We will let $\langle \cdot , \cdot \rangle$ denote the
pairing between $H^{-1}$ and $H^1_0$.
\begin{definition}\label{def:weak}
We say a function
\begin{equation*}
\varphi \in L^2(0,T;H_0^1(0,1)), \text{ with } \varphi_t \in L^2(0,T;H^{-1}(0,1)),
\end{equation*}
is a \emph{weak solution} of (\ref{eq:flow}) with $\varphi_0\in L^2(0,1)$ provided
\begin{align*}
&\langle \varphi_t, v \rangle + \delta(\varphi_x, v_x) = (-\delta^{-1}V'(\varphi)+\alpha,v) \\
\intertext{for each $v\in H_0^1(0,1)$ and a.e.\ time $t_0 \leq t \leq
  T$, and}
&\varphi(t_0)=\varphi_0.
\end{align*}
\end{definition} 
Weak solutions are in fact more regular than is required in the
definition when the initial state $\varphi_0\in H_0^1(0,1)$, which is
used when proving the following theorem.
\begin{thm}
There exists a unique weak solution $\varphi$ to (\ref{eq:flow}) in \linebreak
$C([t_0,T];H_0^1)$ when $\varphi_0 \in H^1_0(0,1)$. 
%This solution
%satisfies $\varphi \in L^2(0,T;H^2)$ and $\varphi_t \in L^2(0,T;L^2)$.
This solution satisfies
\begin{equation}\label{eq:phibound1}
||\varphi_x(t)||^2_{L^2} 
\leq
||(\varphi_0)_x||^2_{L^2}+\frac{\delta^{-2}}{2} ||\varphi_0||^4_{L^4}-
\delta^{-2} ||\varphi_0||^2_{L^2} + \delta^{-1} ||\alpha||^2_{L^2(0,T;L^2)} +\frac{\delta^{-1}}{2},
\end{equation}
for all $t \in [t_0,T]$.
\end{thm} 
\begin{proof}
We start by proving existence and uniqueness of solutions to the equation
(\ref{eq:flow}) when the potential $\tilde V$ is used; see figure
\ref{fig:V}. Let $\breve\varphi \in L^\infty (t_0,T;H_0^1)$ and $\breve\varphi(t_0)=\varphi_0$ and
define $\tilde \varphi$ by 
\begin{equation*}
\tilde\varphi_t=\delta\tilde\varphi_{xx} - \delta^{-1}\tilde V'(\breve\varphi) + \alpha, \quad \tilde\varphi(t_0)=\varphi_0.
\end{equation*}
The solution then satisfies $\tilde\varphi \in L^\infty (t_0,T;H_0^1)$, so
we can define a map
\begin{align*}
A: & L^\infty (t_0,T;H_0^1) \rightarrow L^\infty (t_0,T;H_0^1) \\
& \breve\varphi \mapsto \tilde\varphi
\end{align*} 
which is single valued (see \cite{Evans}). It is now shown that $A$ is a
contraction on $L^\infty (t_0,T;H_0^1)$ if $T$ is small enough. Let
$\tilde \varphi=A(\breve\varphi)$ and $\tilde \psi= A(\breve\psi)$. Subtracting the
equations for $\tilde \varphi$ and $\tilde \psi$ gives
\begin{equation*}
(\tilde \varphi - \tilde \psi)_t=\delta(\tilde \varphi - \tilde \psi)_{xx}
  -\delta^{-1}(\tilde V'(\breve\varphi) - \tilde V'(\breve\psi)),
\end{equation*}
which entails
\begin{equation*}
||\tilde \varphi- \tilde \psi||_{L^\infty(t_0,T;H_0^1)} \leq K ||\tilde
  V'(\breve\varphi)- \tilde V'(\breve\psi)||_{L^2(t_0,T;L^2)},
\end{equation*}
where the constant $K$ decreases when $T$ decreases (see
\cite{Evans}). The right hand side in the previous inequality may be
estimated as
\begin{align*}
& ||\tilde
  V'(\breve\varphi)- \tilde V'(\breve\psi)||_{L^2(t_0,T;L^2)} \leq ||\tilde
  V''||_{L^\infty} ||\breve\varphi - \breve\psi||_{L^2(t_0,T;L^2)} \\
& \leq ||\tilde
  V''||_{L^\infty} ||\breve\varphi - \breve\psi||_{L^2(t_0,T;H_0^1)} \leq ||\tilde
  V''||_{L^\infty} \sqrt{T-t_0} ||\breve\varphi - \breve\psi||_{L^\infty(t_0,T;H_0^1)},
\end{align*}
so that $A$ is a contraction when $T$ is small enough. By splitting
the interval $[t_0,T]$ into smaller subintervals and using the
contraction property on each such interval we obtain the existence and
uniqueness of solutions to (\ref{eq:flow}) when the potential $\tilde
V$ is used. There exists  a continuous representative of solutions to \eqref{eq:flow} in the
equivalence class in $L^\infty(t_0,T;H_0^1)$ (see \cite{Evans} again)
which we call $\varphi$. Since the solution lives in one space dimension
it is continuous as a function of both space and time. So for each
$M>||\varphi_0||$ there is a time $T^*$ such that $||\varphi(t)||_{C(0,1)}
<M$ for all $t \leq T^*$. Thus, in a certain time interval the solution
$\varphi$ is only affected by the unchanged potential $V$ (it never
touches the level where $V$ changes into $\tilde V$). Consider a
time in this interval, and take the inner product with $\varphi_t$ in
\eqref{eq:flow} to get (using $V'(\varphi)=\varphi^3-\varphi$):
\begin{equation*}
||\varphi_t||^2_{L^2} + \frac{\delta}{2}\frac{d}{dt}||\varphi_x||^2_{L^2} +
  \frac{\delta^{-1}}{4}\frac{d}{dt}||\varphi||^4_{L^4}-
  \frac{\delta^{-1}}{2}\frac{d}{dt}||\varphi||^2_{L^2} 
  \leq \half ||\alpha||^2_{L^2}+ \half ||\varphi_t||^2_{L^2}.
\end{equation*}
The $||\varphi_t||^2_{L^2}$ terms are dropped and the resulting
inequality is integrated from $t_0$ to $T^*$:
\begin{multline*}
\delta
||\varphi_x(T^*)||^2_{L^2}+\frac{\delta^{-1}}{2}||\varphi(T^*)||^4_{L^4}-\delta^{-1}
||\varphi(T^*)||^2_{L^2}\\ 
\leq \delta
||(\varphi_0)_x||^2_{L^2}+\frac{\delta^{-1}}{2} ||\varphi_0||^4_{L^4}-
\delta^{-1} ||\varphi_0||^2_{L^2} + \int_{t_0}^{T^*} ||\alpha||^2_{L^2} dt.
\end{multline*}
It is now used that
\begin{equation*}
\frac{\delta^{-1}}{2}||\varphi(T^*)||^4_{L^4}-\delta^{-1}
||\varphi(T^*)||^2_{L^2} =
\delta^{-1}\int_0^1\big(\frac{\varphi(x,T^*)^4}{2}-\varphi(x,T^*)^2\big)dx \geq -\half
\end{equation*}
so that the previous inequality implies
\begin{equation}\label{eq:phibound2}
\delta
||\varphi_x(T^*)||^2_{L^2} 
\leq \delta
||(\varphi_0)_x||^2_{L^2}+\frac{\delta^{-1}}{2} ||\varphi_0||^4_{L^4}-
\delta^{-1} ||\varphi_0||^2_{L^2} + \int_{t_0}^{T^*} ||\alpha||^2_{L^2} dt +\half.
\end{equation}
By Sobolev's inequality we thereby obtain a bound on  the continuous
function $\varphi(T^*)$ in the supremum norm.
Consequently, for all controls $\alpha\in L^2(t_0,T;L^2)$ it is possible
to choose the border point $s$ in Figure \ref{fig:V} between $V$ and
$\tilde V$ large enough so that the solution $\varphi$ is affected only
by the unchanged potential $V$. Note also that it is possible to
choose $T^*=T$ in \eqref{eq:phibound2}. Such a solution is a weak solution to
\eqref{eq:flow} with the original potential $V$. It is unique in
$C(t_0,T;H^1_0)$, for non-uniqueness would otherwise also hold for some
modified potential $\tilde V$. The error bound \eqref{eq:phibound1}
follows from \eqref{eq:phibound2}.
\end{proof}

%Include all entries in the bibliography file.
%\nocite{*}

% Bibliography style.
\bibliographystyle{plain}

% Bib-file
\bibliography{references} 

%\begin{thebibliography}{99}
%\bibitem{young}
%L.C. Young, Lectures on the Calculus of Variations and Optimal Control Theory. 
%Saunders Co., Philadelphia-London-Toronto, Ont. 1969

%\end{thebibliography}
\end{document}